\documentclass[a4paper,english]{article}
\usepackage[a4paper]{geometry}
\usepackage{amsmath,amsthm,amssymb,mathtools,fontenc}
\usepackage{enumerate,csquotes,verbatim}
\usepackage{setspace}
\usepackage[numbers]{natbib}
\usepackage[T1]{fontenc}
\usepackage{hyperref}
\usepackage{color,xcolor}
\usepackage{tikz}
\usetikzlibrary{arrows}
\usepackage{multirow}
\usepackage{graphicx}
\usepackage{tabularx}
\usepackage{dutchcal}
\usepackage{float} 		%%%%% locate table 
\usepackage{tikz}  

\usepackage{subcaption} %%%% in table
\usepackage{colortbl} %%%% in table

\usepackage{diagbox}
\restylefloat{table}    %%%%% big H instrad of h! %%%%%

\newtheorem{thm}{Theorem}
\newtheorem{lem}[thm]{Lemma}

\newtheorem{prop}[thm]{Proposition}

\theoremstyle{definition}

\newtheorem{prob}[thm]{Problem}
\newtheorem{conj}[thm]{Conjecture}

\theoremstyle{claim}
\newtheorem*{claim}{Claim}

\theoremstyle{main}

\title{Multithreshold multipartite graphs with small parts}
\author{Teeradej Kittipassorn\thanks{\,Department of Mathematics and Computer Science, Faculty of Science, Chulalongkorn University, Bangkok 10330, Thailand; \texttt{teeradej.k@chula.ac.th}.}
  \and Thanaporn Sumalroj\thanks{\,Department of Mathematics and Computer Science, Faculty of Science, Chulalongkorn University, Bangkok 10330, Thailand; \texttt{t.sumalroj@gmail.com}.}}
\date{}

\begin{document}
	
		\maketitle 
		
\begin{abstract}
A graph is a \emph{$k$-threshold} graph with \emph{thresholds} $\theta_1, \theta_2, \dots, \theta_k$ if we can assign a real number $r_v$ to each vertex $v$ such that for any two distinct vertices $u$ and $v$, 
$uv$ is an edge if and only if the number of thresholds not exceeding $r_u+r_v$ is odd.	
The \emph{threshold number} of a graph is the smallest $k$ for which it is a $k$-threshold graph.
Multithreshold graphs were introduced by Jamison and Sprague as a generalization of classical threshold graphs.
They asked for the exact threshold numbers of complete multipartite graphs. 
Recently, Chen and Hao solved the problem for complete multipartite graphs where each part is not too small, and they asked for the case when each part has size $3$.
We determine the exact threshold numbers of $K_{3, 3, \dots, 3}$, $K_{4, 4, \dots, 4}$ and their complements $nK_3$, $nK_4$.
This improves a result of Puleo.
\end{abstract}	
	
%%%%%%%%%%%%%%%%%%%%%%%%%%%%%%%%%%%%%%%%%%%%%%%%%%%%%%%%%%%%%%%%%%%%%%%%%%%%%%%%%%%%%%
	
\section{Introduction}

A graph $G$ is said to be a \emph{threshold graph} if we can assign a real number $r_v$ to each vertex $v$ and there is a real number $\theta$ such that for any vertex subset $U$ of $G$, $\sum_{v\in U}r_v\leq \theta$  if and only if $U$ is independent in $G$. 
The class of threshold graphs was first introduced by Chv\'{a}tal and Hammer~\cite{C} in 1977 to answer a question in integer linear programming.
As one of the fundamental classes of graphs, properties of threshold graphs have been extensively studied (see~\cite{GJ, HMP, HZ, MP, OJ}), and since then many applications of these graphs have been found in various areas (see~\cite{CMP, G, K, OE, PS}).

Threshold graphs can be characterized in a number of equivalent ways.
For example, $G$ is a threshold graph if and only if $G$ has no induced subgraph isomorphic to $2K_2, P_4$ or $C_4$ (see~\cite{C, MP}). 
Equivalently, a threshold graph is a graph that can be obtained from the single-vertex graph by repeatedly adding an isolated vertex or a universal vertex (see~\cite{C, MP}).  
Moreover, $G$ is a threshold graph if and only if we can assign a real number $r_v$ to each vertex $v$ and there is a real number $\theta$ such that for any two distinct vertices $u$ and $v$, $uv$ is an edge if and only if $r_u+r_v \geq \theta$ (see~\cite{MP}).

Recently, Jamison and Sprague~\cite{JS} introduced \emph{multithreshold graphs} as a generalization of the well-studied threshold graphs as follows.
A graph $G$ is a \emph{$k$-threshold} graph with \emph{thresholds} $\theta_1, \theta_2, \dots, \theta_k$ if we can assign a real number $r_v$, called a \emph{rank}, to each vertex $v$ such that for any two distinct vertices $u$ and $v$, 
$uv$ is an edge if and only if the number of thresholds not exceeding $r_u+r_v$ is odd.
Equivalently,
\begin{align*}
	uv\in E(G) \iff r_u + r_v\in\left[\theta_{2i-1},\theta_{2i}\right) \hspace{0.1cm}\text{for some} \hspace{0.1cm} i\in \left\lbrace 1, 2,\dots, \left\lceil{\frac{k}{2}} \right\rceil\right\rbrace   
\end{align*}
provided $\theta_1<\theta_2<\cdots< \theta_k$ and $\theta_{k+1}=\infty$.
We call such an assignment $r$ of ranks a $\left(\theta_1, \theta_2,\dots, \theta_k \right)$-\emph{representation of $G$}.
By a \emph{rank sum of an edge/nonedge} $uv$, we mean $r_u+r_v$.
Note that the case of one threshold agrees with the classical threshold graphs introduced by Chv\'{a}tal and Hammer~\cite{C}.

Jamison and Sprague~\cite{JS} showed that any graph  of order $n$ is a $k$-threshold graph for some $k\leq \binom{n}{2}$. 
The smallest $k$ for which a graph $G$ is a $k$-threshold graph is said to be the \emph{threshold number} of $G$, denoted by $\Theta(G)$.
Observe that $\Theta(H)\leq\Theta(G)$ for any induced subgraph $H$ of $G$.
In addition, they found a relationship between the threshold numbers of a graph and its complement.
We give a proof for completeness.
%%%%%%%%%%%%%%%%%%%%%%%%%%%%%%%%%%%%%%%%%%%
\begin{prop}[Jamison and Sprague~\cite{JS}] \label{complement} 
For any graph $G$, either 	
\begin{equation*}
	\displaystyle\Theta(G^c)=\Theta(G) \ \text{or} \  \left\lbrace \Theta(G), \displaystyle\Theta(G^c) \right\rbrace = \left\lbrace 2k, 2k+1 \right\rbrace \  \text{for some} \ k\in\mathbb{N}.
\end{equation*}
\end{prop}

\begin{proof}
Let $k$ and $k'$ be such that $\Theta(G)\in\left\lbrace 2k, 2k+1 \right\rbrace$ and $\displaystyle\Theta(G^c)\in\left\lbrace 2k', 2k'+1 \right\rbrace$.
Take a $\left(\theta_1, \theta_2,\dots, \theta_{\Theta(G)}\right)$-representation $r$ of $G$.	
We may assume that no rank sum equals a threshold by perturbing each threshold to the left.
We may further assume that $r$ has $2k+1$ thresholds by adding a sufficiently large threshold $\theta_{2k+1}$ if neccesary.
We then obtain a $\left(-\theta_{2k+1}, -\theta_{2k},\dots, -\theta_1\right)$-representation of $\displaystyle G^c$ from $r$ by reversing the values of the ranks and the thresholds.
Thus $\displaystyle\Theta(G^c)\leq2k+1$, and so $k'\leq k$.
Similarly, $\Theta(G)\leq2k'+1$, and so $k\leq k'$.
Now, we have $k=k'$, and hence  
$\Theta(G), \displaystyle\Theta(G^c) \in \left\lbrace 2k, 2k+1 \right\rbrace$. 
\end{proof}
%%%%%%%%%%%%%%%%%%%%%%%%%%%%%%%%%%%%%%%%%%%
This inspired Jamison and Sprague~\cite{JS} to put forward the following conjecture. 
%%%%%%%%%%%%%%%%%%%%%%%%%%%%%%%%%%%%%%%%%%%
\begin{conj}[Jamison and Sprague~\cite{JS}] \label{conj} 
	For all $k\geq 1$, there is a graph $G$ with $\Theta(G)=2k$ and $\displaystyle\Theta(G^c)=2k+1$. 
\end{conj}
%%%%%%%%%%%%%%%%%%%%%%%%%%%%%%%%%%%%%%%%%%%
We denote by $K_{m_1, m_2, \dots, m_n}$ the complete $n$-partite graph with $m_i$ vertices in the $i^{th}$ part for $1\leq i\leq n$, and by
$K_{n\times m}$  the complete $n$-partite graph with $m$ vertices in each part.
Jamison and Sprague~\cite{JS} observed that, by assigning the rank $3^i$ to each vertex of the $i^{th}$ part of $K_{m_1, m_2, \dots, m_n}$, the threshold number is at most $2n$, and they left the following problem.
%%%%%%%%%%%%%%%%%%%%%%%%%%%%%%%%%%%%%%%%%%%
\begin{prob}[Jamison and Sprague~\cite{JS}] \label{prob}
	Determine the exact threshold number of $K_{m_1, m_2, \dots, m_n}$.
\end{prob} 
%%%%%%%%%%%%%%%%%%%%%%%%%%%%%%%%%%%%%%%%%%%
Chen and Hao~\cite{CH} recently gave a partial solution of Problem~\ref{prob} which also confirmed Conjecture~\ref{conj}.
%%%%%%%%%%%%%%%%%%%%%%%%%%%%%%%%%%%%%%%%%%%
\begin{thm}[Chen and Hao~\cite{CH}]\label{CH} 
	Let $m_1, m_2, \dots, m_n$ be $n\geq 2$ positive integers. 
	If $m_i\geq n+1$ for $i=1, 2,\dots, n$, then 
	\begin{equation*}
		\Theta(K_{m_1, m_2, \dots, m_n})=2n-2 \ \ \text{and} \ \ \displaystyle\Theta(K_{m_1, m_2, \dots, m_n}^c)=2n-1.
	\end{equation*}
\end{thm} 
%%%%%%%%%%%%%%%%%%%%%%%%%%%%%%%%%%%%%%%%%%%
However, their result is far from the truth when $m_i$ are small. For instance, it is readily seen that the threshold number of $K_{n\times 1}$ is $1$.  
Jamison and Sprague~\cite{JS} observed that the threshold number of $K_{n\times 2}$ is at most $3$ by making the nonedge rank sum in each part equal.
Chen and Hao~\cite{CH} mentioned that it would be interesting to know the value of $\Theta(K_{n\times 3})$.
As a tool for answering a question of Jamison, Puleo~\cite{P} proved that $\Theta(nK_3)\geq n^{1/3}$,
which in turn provides a lower bound for $K_{n\times 3}$ by Proposition~\ref{complement}.
%%%%%%%%%%%%%%%%%%%%%%%%%%%%%%%%%%%%%%%%%%%

In this paper, our main results determine the exact threshold numbers of $K_{n\times 3}$, $K_{n\times 4}$ and their complements $nK_3$, $nK_4$.
The theorems also give more examples satisfying Conjecture~\ref{conj}.
%%%%%%%%%%%%%%%%%%%%%%%%%%%%%%%%%%%%%%%%%%%
\begin{thm} \label{Main theorem K3}
	Let $p_m=m+\binom{m}{3}+2$ and $q_m=m+\binom{m}{3}+1$.
	\begin{enumerate}[(i)]
		\item For $n\geq 2$, 
			\begin{equation*}
			\Theta(K_{n\times3})=
			\left\{
			\begin{array}{l}
				2m \hspace{.8cm}\text{if} \hspace{.2cm} n=p_{m-1},
				\\
				2m+1 \hspace{.2cm}\text{if} \hspace{.2cm} p_{m-1} < n < p_m.
			\end{array}
			\right.
		\end{equation*}	
		\item For $n\geq 1$,
		\begin{equation*}
			\Theta(nK_3)=
			\left\{
			\begin{array}{l}
				2m-1 \hspace{.2cm}\text{if} \hspace{.2cm} n=q_{m-1},
				\\
				2m \hspace{.8cm}\text{if} \hspace{.2cm} q_{m-1} < n < q_m.
			\end{array}
			\right.
		\end{equation*}	
	\end{enumerate}
\end{thm}
%%%%%%%%%%%%%%%%%%%%%%%%%%%%%%%%%%%%%%%%%%%
\begin{thm} \label{Main theorem K4}
	Let $s_m=m+\binom{\left\lfloor{m/2}\right\rfloor}{3}+\binom{\left\lceil{m/2}\right\rceil}{3}+2$ and $t_m=m+\binom{\left\lfloor{m/2}\right\rfloor}{3}+\binom{\left\lceil{m/2}\right\rceil}{3}+1$. 
	\begin{enumerate}[(i)]
		\item For $n\geq 2$, 
		\begin{equation*}
		\Theta(K_{n\times4})=
		\left\{
		\begin{array}{l}
			2m \hspace{.8cm}\text{if} \hspace{.2cm} n=s_{m-1},
			\\
			2m+1 \hspace{.2cm}\text{if} \hspace{.2cm} s_{m-1} < n < s_m.
		\end{array}
		\right.
	\end{equation*}	
	\item For $n\geq 1$, 
	\begin{equation*}
		\Theta(nK_4)=
		\left\{
		\begin{array}{l}
			2m-1 \hspace{.2cm}\text{if} \hspace{.2cm} n=t_{m-1},
			\\
			2m \hspace{.8cm}\text{if} \hspace{.2cm} 
			t_{m-1}< n < t_m.
		\end{array}
		\right.
	\end{equation*}	
	\end{enumerate}
\end{thm}
%%%%%%%%%%%%%%%%%%%%%%%%%%%%%%%%%%%%%%%%%%%

The proofs of the lower bounds for the threshold numbers are based on an idea of Puleo~\cite{P}.
We assign a color to each edge of $nK_3$ and $nK_4$ according to the location of its rank sum with respect to the thresholds.
For $nK_3$, we show that colors $ijj$ and $i\ell\ell$ cannot appear on two triangles simultaneously.	
For $nK_4$, we prove that each $K_4$ must contain a particular kind of $K_3$.
These help us find the maximum value of $n$ in terms of the number of colors.

The rest of this paper is organized as follows. In Section~\ref{scetion-K3andKnx3}, we prove Theorem~\ref{Main theorem K3}, and
Section~\ref{scetion-K4andKnx4} is devoted to proving Theorem~\ref{Main theorem K4}.
Finally, we give some concluding
remarks in Section~\ref{section-conclusion}.
%%%%%%%%%%%%%%%%%%%%%%%%%%%%%%%%%%%%%%%%%%%
%%%%%%%%%%%%%%%%%%%%%%%%%%%%%%%%%%%%%%%%%%%

\section{Threshold numbers of $K_{n\times3}$ and $nK_3$} \label{scetion-K3andKnx3}

In this section, we determine the values of $\Theta(K_{n\times3})$ and $\Theta(nK_3)$.
To outline the proofs, we will need five lemmas.
Lemmas~\ref{NoTwoSameColor} to~\ref{nocolorK3} are for the lower bounds where the key idea is in Lemma~\ref{KeyLowerBound}.
We apply Lemmas~\ref{NoTwoSameColor} and~\ref{KeyLowerBound} to prove Lemma~\ref{Max.of.n}, which determines the maximum number of triangles and parts in terms of the number of colors.
Lemma~\ref{nocolorK3} helps improve the lower bounds obtained from Lemma~\ref{Max.of.n}.
On the other hand, Lemma~\ref{EnotNE} is a tool to prove the upper bounds.  

We start by assigning a color to each edge of $nK_3$ and each nonedge of $K_{n\times3}$ as follows.
In a $\left(\theta_1, \theta_2,\dots, \theta_k \right)$-representation of $nK_3$ where $\theta_1<\theta_2<\cdots< \theta_k$, we color an edge $uv$ with \emph{color $i$}, for $i\in \left\lbrace 1, 2,\dots, \left\lceil{\frac{k}{2}} \right\rceil\right\rbrace$, if $r_u + r_v\in\left[\theta_{2i-1},\theta_{2i}\right)$.
We say that a triangle has a \emph{color} $i j \ell$ if the colors appearing on its edges are $i, j$ and $\ell$.

Similarly, in a $\left(\theta_1, \theta_2,\dots, \theta_k \right)$-representation of $K_{n\times3}$ where $\theta_1<\theta_2<\cdots< \theta_k$, we color a nonedge $xy$ with \emph{color $i$}, for $i\in \left\lbrace 1, 2,\dots, \left\lceil{\frac{k+1}{2}} \right\rceil\right\rbrace$, if $r_x + r_y\in\left[\theta_{2i-2},\theta_{2i-1}\right)$ where $\theta_0=-\infty$.
We say that a part has a \emph{color} $i j \ell$ if the colors appearing on its nonedges are $i, j$ and $\ell$.

First, we need a result of Puleo~\cite{P} which says that no two triangles in $nK_3$ have the same color.
Interchanging edges and nonedges, no two parts in $K_{n\times3}$ have the same color.
We include a proof for completeness.
%%%%%%%%%%%%%%%%%%%%%%%%%%%%%%%%%%%%%%%%%%%
\begin{lem}[Puleo~\cite{P}] \label{NoTwoSameColor} 
	\begin{enumerate} [(i)]
		\item In a $\left(\theta_1, \theta_2,\dots, \theta_k \right)$-representation of $nK_3$, no two triangles have the same color.
		\item In a $\left(\theta_1, \theta_2,\dots, \theta_k \right)$-representation of $K_{n\times3}$, no two parts have the same color.
	\end{enumerate}
\end{lem}

\begin{proof}
	We will only prove $(i)$ as the proof of $(ii)$ is similar.
	Let $r$ be a $\left(\theta_1, \theta_2,\dots, \theta_k \right)$-representation of $nK_3$ where $\theta_1<\theta_2<\cdots< \theta_k$.
	Suppose to the contrary that there are two triangles $T_x$ and $T_y$ in $nK_3$ having the same color $ij\ell$. 
	Thus if $V(T_x)=\{x_1, x_2, x_3\}$ and  $V(T_y)=\{y_1, y_2, y_3\}$, then their edge rank sums are as follows:
	\begin{align*}
		r_{x_1}+r_{x_2}, r_{y_1}+r_{y_2}&\in\left[\theta_{2i-1},\theta_{2i}\right),\\
			r_{x_1}+r_{x_3}, r_{y_1}+r_{y_3}&\in\left[\theta_{2j-1},\theta_{2j}\right), \\  
		r_{x_2}+r_{x_3}, r_{y_2}+r_{y_3}&\in\left[\theta_{2\ell-1},\theta_{2\ell}\right).
	\end{align*}
Note that at least two ranks out of $\max\left\lbrace r_{x_1}, r_{y_1} \right\rbrace, \max\left\lbrace r_{x_2}, r_{y_2} \right\rbrace, \max\left\lbrace r_{x_3}, r_{y_3} \right\rbrace$ are from the same triangle.
Without loss of generality, let $r_{x_1}\leq r_{y_1}$ and $r_{x_2}\leq r_{y_2}$.
Write $r_{x_p}=\min\left\lbrace r_{x_1}, r_{x_2} \right\rbrace$ and $r_{y_q}=\max\left\lbrace r_{y_1}, r_{y_2} \right\rbrace$.
Then $\theta_{2i-1}\leq r_{x_1}+r_{x_2}\leq r_{x_p}+r_{y_q}\leq r_{y_1}+r_{y_2}<\theta_{2i} $.
By the definition of thresholds, $x_py_q$ is an edge of color $i$, which contradicts the fact that $x_py_q$ is a nonedge in $nK_3$. 
\end{proof} 	 
%%%%%%%%%%%%%%%%%%%%%%%%%%%%%%%%%%%%%%%%%%%
The next lemma is the key idea for obtaining the lower bounds for the threshold numbers.
%%%%%%%%%%%%%%%%%%%%%%%%%%%%%%%%%%%%%%%%%%%
\begin{lem} \label{KeyLowerBound}
	\begin{enumerate} [(i)]
	\item In a $\left(\theta_1, \theta_2,\dots, \theta_k \right)$-representation of $nK_3$ and colors $i, j, \ell \in \left\lbrace 1, 2,\dots, \left\lceil{\frac{k}{2}} \right\rceil\right\rbrace$, colors $ijj$ and $i\ell\ell$ cannot appear on two triangles simultaneously.	
	\item 	In a $\left(\theta_1, \theta_2,\dots, \theta_k \right)$-representation of $K_{n\times3}$ and colors $i, j, \ell \in \left\lbrace 1, 2,\dots, \left\lceil{\frac{k+1}{2}} \right\rceil\right\rbrace$, colors $ijj$ and $i\ell\ell$ cannot appear on two parts simultaneously.	
\end{enumerate}
\end{lem}

\begin{proof} 
	We will only prove $(i)$ as the proof of $(ii)$ is similar.
	Let $r$ be a $\left(\theta_1, \theta_2,\dots, \theta_k \right)$-representation of $nK_3$ where $\theta_1<\theta_2<\cdots< \theta_k$.
	Suppose to the contrary that there are two triangles $T_x$ and $T_y$ in $nK_3$ of colors $ijj$ and $i\ell\ell$ respectively. 
	Thus if $V(T_x)=\{x_1, x_2, x_3\}$ and  $V(T_y)=\{y_1, y_2, y_3\}$, then their edge rank sums are as follows:
	\begin{align*}
		 a_1&=r_{x_1}+r_{x_3}\in\left[\theta_{2i-1},\theta_{2i}\right), &b_1&=r_{x_1}+r_{x_2}\in\left[\theta_{2j-1}, \theta_{2j}\right), &b_2&=r_{x_2}+r_{x_3}\in\left[\theta_{2j-1}, \theta_{2j}\right), 
		 \\
		 a_2&=r_{y_1}+r_{y_3}\in\left[\theta_{2i-1}, \theta_{2i}\right), &c_1&=r_{y_1}+r_{y_2}\in\left[\theta_{2\ell-1}, \theta_{2\ell}\right), &c_2&=r_{y_2}+r_{y_3}\in\left[\theta_{2\ell-1}, \theta_{2\ell}\right).
	\end{align*}
From these rank sums, we can compute the ranks as follows:
\begin{align*}
	r_{x_1}&=\frac{a_1+b_1-b_2}{2},\ &r_{x_2}&=\frac{b_1+b_2-a_1}{2},\ &r_{x_3}&=\frac{a_1+b_2-b_1}{2}, 
	\\
	r_{y_1}&=\frac{a_2+c_1-c_2}{2},\ &r_{y_2}&=\frac{c_1+c_2-a_2}{2},\ &r_{y_3}&=\frac{a_2+c_2-c_1}{2}.
\end{align*}
	Without loss of generality, let $a_1\leq a_2$, $b_1\leq b_2$ and $c_1\leq c_2$.
	Let $D=a_2-a_1\geq0$ and let
\begin{align*}
	A&=b_1-b_2+c_1-c_2,\\
	B&=-b_1+b_2+c_1-c_2,\\  
	C&=-b_1+b_2-c_1+c_2.
\end{align*}
	Note that $A\leq B\leq C$ and $A\leq 0$.
	Since $D\geq 0\geq A$, either $D\in [A, B]$, $D\in [B, C]$ or $D\in [C, \infty)$.
	We obtain a contradiction by the following three claims.
\begin{claim}
	$D\notin [A, B]$.
\end{claim}
	Since $x_2y_3$ is a nonedge, we cannot have $b_1\leq r_{x_2}+r_{y_3}\leq b_2$;  otherwise, $r_{x_2}+r_{y_3}\in\left[\theta_{2j-1}, \theta_{2j}\right)$.
	Observe that
\begin{equation*}
	\begin{aligned}
		b_1\leq r_{x_2}+r_{y_3}\leq b_2 & \iff b_1\leq \frac{b_1+b_2-a_1}{2}+\frac{a_2+c_2-c_1}{2}\leq b_2\\
		& \iff b_1-b_2+c_1-c_2\leq a_2-a_1\leq -b_1+b_2+c_1-c_2 \\
		& \iff A\leq D \leq B.
	\end{aligned}
\end{equation*}
\begin{claim}
$D\notin [B, C]$.
\end{claim}
Since $x_3y_2$ is a nonedge, we cannot have $c_1\leq r_{x_3}+r_{y_2}\leq c_2$;  otherwise, $r_{x_3}+r_{y_2}\in\left[\theta_{2\ell -1}, \theta_{2\ell}\right)$.
Note that
\begin{equation*}
	\begin{aligned}
		c_1\leq r_{x_3}+r_{y_2}\leq c_2 & \iff c_1\leq \frac{a_1+b_2-b_1}{2}+\frac{c_1+c_2-a_2}{2}\leq c_2\\
		& \iff -b_1+b_2+c_1-c_2\leq a_2-a_1\leq -b_1+b_2-c_1+c_2 \\
		& \iff B\leq D \leq C.
	\end{aligned}
\end{equation*}
\begin{claim}
$D\notin [C, \infty)$.
\end{claim}
Since $x_3y_3$ is a nonedge, we cannot have $a_1\leq r_{x_3}+r_{y_3}\leq a_2$;  otherwise, $r_{x_3}+r_{y_3}\in\left[\theta_{2i-1}, \theta_{2i}\right)$.
Observe that
\begin{align*}
		a_1\leq r_{x_3}+r_{y_3}\leq a_2 & \iff a_1\leq \frac{a_1+b_2-b_1}{2}+\frac{a_2+c_2-c_1}{2}\leq a_2\\
		& \iff a_1-a_2\leq-b_1+b_2-c_1+c_2\leq a_2-a_1 \\
		& \iff -D\leq C \leq D.\qedhere
	\end{align*}
\end{proof}	
%%%%%%%%%%%%%%%%%%%%%%%%%%%%%%%%%%%%%%%%%%%
We apply Lemmas~\ref{NoTwoSameColor} and~\ref{KeyLowerBound} to determine the maximum number of triangles and parts in terms of the number of colors, which in turn gives lower bounds for the threshold numbers.
%%%%%%%%%%%%%%%%%%%%%%%%%%%%%%%%%%%%%%%%%%%
\begin{lem} \label{Max.of.n}
\begin{enumerate}[(i)]
	\item If there are at most $m$ colors of edges in $nK_3$, then $n\leq m +\binom{m}{3}$.	
	In particular, if $nK_3$ is a $k$-threshold graph, then $n\leq \left\lceil{\frac{k}{2}}\right\rceil +\binom{\left\lceil k/2 \right\rceil}{3}$.	
	\item If there are at most $m$ colors of nonedges in $K_{n\times3}$, then $n\leq m +\binom{m}{3}$.	
	In particular, if $K_{n\times3}$ is a $k$-threshold graph, then $n\leq \left\lceil{\frac{k+1}{2}}\right\rceil +\binom{\left\lceil (k+1)/2 \right\rceil}{3}$.	
\end{enumerate}	
\end{lem} 

\begin{proof} 
	We will only prove $(i)$ as the proof of $(ii)$ is similar.
	Suppose that there are at most $m$ colors of edges in $nK_3$.
	By Lemma~\ref{NoTwoSameColor}, there are at most $\binom{m}{3}$ triangles in $nK_3$ whose edges are colored with $3$ colors.
	It is sufficient to show that there are at most $m$ triangles in $nK_3$ whose edges are colored with $1$ or $2$ colors. 
	Indeed, for each color $i$, there is at most one triangle of color of the form $i j j$ where $j\in [m]$ by Lemma~\ref{KeyLowerBound}.
	Thus $n\leq m +\binom{m}{3}$.	
	Note that if $nK_3$ is a $k$-threshold graph, then there are at most $\left\lceil\frac{k}{2}\right\rceil$ colors of edges in $nK_3$, and
	hence, $n\leq \left\lceil{\frac{k}{2}}\right\rceil +\binom{\left\lceil k/2 \right\rceil}{3}$.	
\end{proof}	
%%%%%%%%%%%%%%%%%%%%%%%%%%%%%%%%%%%%%%%%%%%
The lower bounds for the threshold numbers obtained from Lemma~\ref{Max.of.n} are not sharp. We require another observation which states roughly that the first and last colors appear in at most one triangle or part.
%%%%%%%%%%%%%%%%%%%%%%%%%%%%%%%%%%%%%%%%%%%
\begin{lem} \label{nocolorK3} 
	\begin{enumerate}[(i)]
		\item In a $\left(\theta_1, \theta_2,\dots, \theta_{2m-1}\right)$-representation of $nK_3$, an edge of color $m$ appears in at most one triangle.
		\item In a $\left(\theta_1, \theta_2,\dots, \theta_m\right)$-representation of $K_{n\times3}$, a nonedge of color $1$ appears in at most one part.
		\item In a $\left(\theta_1, \theta_2,\dots, \theta_{2m}\right)$-representation of $K_{n\times3}$, a nonedge of color $m+1$ appears in at most one part.
	\end{enumerate}
\end{lem}

\begin{proof}
	We only prove $(i)$ as the proofs of $(ii)$ and $(iii)$ are similar. 
	Let $r$ be a $\left(\theta_1, \theta_2,\dots, \theta_{2m-1} \right)$-representation of $nK_3$.
	Suppose to the contrary that there are two triangles $T_x$ and $T_y$ in $nK_3$ with an edge of color $m$. 
	Let $x_1, x_2, x_3$ and $y_1, y_2, y_3$ be vertices of $T_x$ and $T_y$  respectively  where $x_1x_2$ and $y_1y_2$ are edges of color $m$, that is 
	$r_{x_1}+r_{x_2}$, $r_{y_1}+r_{y_2}\geq \theta_{2m-1}$.
	Assume without loss of generality that  $r_{x_1}, r_{y_1}\geq\frac{\theta_{2m-1}}{2}$.
	Thus $r_{x_1}+r_{y_1}\geq\theta_{2m-1}$, which implies that $x_1y_1$ is an edge in $nK_3$, a contradiction.
\end{proof}
%%%%%%%%%%%%%%%%%%%%%%%%%%%%%%%%%%%%%%%%%%%

The upper bounds for the threshold numbers will be obtained by rank assignments of the following forms. 
A rank assignment $r$ of $nK_3$ is said to be an \emph{$\left\lbrace a_1, a_2,\dots, a_m \right\rbrace$-assignment} if each triangle has edge rank sums of the form $a_i, a_i, a_i$ or $a_i, a_j, a_k$ for distinct $i, j, k\in[m]$, and 
no two triangles have the same multiset of edge rank sums.
Note that if $n\leq m+\binom{m}{3}$, then  an $\left\lbrace a_1, a_2,\dots, a_m \right\rbrace$-assignment of $nK_3$ exists since we can assign any edge rank sums for each triangle as a triangle has edge rank sums $a_i, a_j, a_k$ if and only if its ranks are $\frac{a_i+a_j-a_k}{2}, \frac{a_i+a_k-a_j}{2}, \frac{a_j+a_k-a_i}{2}$.

In the same fasion, a rank assignment $r$ of $K_{n\times3}$ is said to be an \emph{$\left\lbrace a_1, a_2,\dots, a_m \right\rbrace$-assignment} if each part has nonedge rank sums of the form $a_i, a_i, a_i$ or $a_i, a_j, a_k$ for distinct $i, j, k\in[m]$, and 
no two parts have the same multiset of nonedge rank sums.
Observe that an $\left\lbrace a_1, a_2,\dots, a_m \right\rbrace$-assignment of $K_{n\times3}$ exists whenever $n\leq m+\binom{m}{3}$.

%%%%%%%%%%%%%%%%%%%%%%%%%%%%%%%%%%%%%%%%%%%
The linear independence of $\left\lbrace a_1, a_2, \dots, a_m \right\rbrace$ over $\mathbb{Q}$ is a sufficient condition for the edge and nonedge rank sums in an $\left\lbrace a_1, a_2,\dots, a_m \right\rbrace$-assignment not to  coincide.  
%%%%%%%%%%%%%%%%%%%%%%%%%%%%%%%%%%%%%%%%%%%
\begin{lem} \label{EnotNE} 
	Let $\left\lbrace a_1, a_2, \dots, a_m \right\rbrace\subset\mathbb{R}$ be a linearly independent set over $\mathbb{Q}$.
\begin{enumerate}[(i)]
	\item In an $\left\lbrace a_1, a_2,\dots, a_m \right\rbrace$-assignment of $nK_3$, the edge and nonedge rank sums do not coincide.
	\item In an $\left\lbrace a_1, a_2,\dots, a_m\right\rbrace$-assignment of $K_{n\times3}$, the edge and nonedge rank sums do not coincide.
\end{enumerate}
\end{lem}

\begin{proof}
$(i)$ Let $\left\lbrace a_1, a_2, \dots, a_m \right\rbrace\subset\mathbb{R}$ be a linearly independent set over $\mathbb{Q}$ and 
let $r$ be the $\left\lbrace a_1, a_2, \dots, a_m \right\rbrace$-assignment of $nK_3$.
Suppose to the contrary that there exists a  nonedge $xy$ in $nK_3$ such that $r_x+r_y=a_\ell$ for some $\ell\in[m]$.
Since each triangle has edge rank sums of the form $a_i, a_j, a_k$ where $i, j, k\in[m]$ are all equal or all distinct, the rank of each vertex is of the form $\frac{a_i+a_j-a_k}{2}$.
Thus we suppose that $r_x=\frac{a_i+a_j-a_k}{2}$ and  $r_y=\frac{a_r+a_s-a_t}{2}$ 
where $i, j, k\in[m]$ are all equal or all distinct, $r, s, t\in[m]$ are all equal or all distinct, and $\left\lbrace i, j, k \right\rbrace \neq \left\lbrace r, s, t \right\rbrace$.
Hence, $r_x+r_y=a_\ell$ becomes
\begin{equation*}
	a_i+a_j-a_k+a_r+a_s-a_t=2a_{\ell}. 
\end{equation*}
Since $\left\lbrace a_i, a_j, a_k \right\rbrace \neq \left\lbrace a_r, a_s, a_t \right\rbrace$, there exists an element in one set not appearing in the other set, say $a_i\notin\left\lbrace a_r, a_s, a_t \right\rbrace$.
Since $i, j, k$ are all equal or all distinct, the coefficient of $a_i$ after simplifying the left hand side of the equation is $1$.
Thus the left hand side cannot equal  $2a_\ell$, a contradiction.

The proof of $(ii)$ is similar.
\end{proof}
%%%%%%%%%%%%%%%%%%%%%%%%%%%%%%%%%%%%%%%%%%%
We are now ready to prove Theorem~\ref{Main theorem K3}$(i)$.
%%%%%%%%%%%%%%%%%%%%%%%%%%%%%%%%%%%%%%%%%%%
%\begin{thm} \label{TK3}
%Let $T_k=k+\binom{k}{3}+1$. Then
%\begin{equation*}
%	\Theta(nK_3)=
%	\left\{
%		\begin{array}{l}
%		 2m-1 \hspace{.2cm}\text{if} \hspace{.2cm} n=T_{m-1},
%		\\
%	   	2m \hspace{.8cm}\text{if} \hspace{.2cm} T_{m-1} < n < T_m.
%		\end{array}
%	\right.
%\end{equation*}	
%\end{thm}

\begin{proof} [Proof of Theorem~\ref{Main theorem K3}$(i)$] 
	
Let $m$ be such that $q_{m-1} \leq n < q_m$. 
Suppose to the contrary that $\Theta(nK_3)\leq 2m-2$. 
By Lemma~\ref{Max.of.n}$(i)$, 
\begin{equation*}
n\leq \left\lceil{\frac{\Theta(nK_3)}{2}}\right\rceil+\binom{\left\lceil\frac{\Theta(nK_3)}{2}\right\rceil}{3} 
\leq \left\lceil{\frac{2m-2}{2}}\right\rceil+
\binom{\left\lceil\frac{2m-2}{2}\right\rceil}{3}
=(m-1)+\binom{m-1}{3}
=q_{m-1}-1
\end{equation*}
contradicting the definition of $m$. Thus $\Theta(nK_3)\geq2m-1$. 

To prove that $\Theta(nK_3)\leq2m$, 
let $A=\left\{a_1, a_2, \dots, a_m\right\}\subset \mathbb{R^{+}}$ be a linearly independent set over $\mathbb{Q}$,
for example, let  $a_i=\sqrt{p_i}$ where $p_i$ is  the $i^{\text{th}}$ prime number.
Since $n\leq q_m-1 = m+\binom{m}{3}$,
we can pick an $A$-assignment for $nK_3$.
By Lemma~\ref{EnotNE}$(i)$, the edge and nonedge rank sums do not coincide. 
We separate the edge and nonedge rank sums by putting two thresholds around each edge rank sum.
For $i=1, 2,\dots, m$, let $\theta_{2i-1}=a_i$ and 
$\theta_{2i}=a_i+\varepsilon$ be thresholds of $nK_3$ where $\varepsilon$ is a sufficiently small positive real number, for example, take $\varepsilon$ smaller than any distance between two distinct rank sums of $nK_3$.
Thus the above rank assignment is a $\left(\theta_1, \theta_2,\dots, \theta_{2m} \right)$-representation of $nK_3$.
Hence, $nK_3$ is a $2m$-threshold graph, that is $\Theta(nK_3)\leq2m$ as desired.

Suppose that $n=q_{m-1}$. 
To prove that $\Theta(nK_3)\leq2m-1$, 
let $A=\left\{a_1, a_2, \dots, a_m\right\}\subset\mathbb{R^{+}}$ be a linearly independent set over $\mathbb{Q}$ such that $a_1<a_2<\dots<a_{m-1}\leq\frac{a_m}{2}$.
We then pick an $A\setminus \left\lbrace a_m \right\rbrace$-assignment for the first $(m-1)+\binom{m-1}{3}$ triangles in $nK_3$, and let the last triangle have edge rank sums $a_m, a_m, a_m$. 
Note that this is an $A$-assignment of $nK_3$.
By Lemma~\ref{EnotNE}$(i)$, the edge and nonedge rank sums do not coincide. 
We separate the edge and nonedge rank sums by putting two thresholds around each edge rank sum.
For $i=1, 2,\dots, m$, let $\theta_{2i-1}=a_i$ and 
$\theta_{2i}=a_i+\varepsilon$ be thresholds of $nK_3$ where $\varepsilon$ is a sufficiently small positive real number. 
Thus the above rank assignment is a $\left(\theta_1, \theta_2,\dots, \theta_{2m} \right)$-representation of $nK_3$.
In fact, we will show that we do not need the last threshold $\theta_{2m}$ by proving that no rank sum exceeds $\theta_{2m-1}$.
It is sufficient to show that the rank of each vertex is at most $\frac{\theta_{2m-1}}{2}=\frac{a_m}{2}$.
This is clear for the last triangle with edge rank sums $a_m, a_m, a_m$.
For the other triangles, the rank of each vertex is of the form $\frac{a_i+a_j-a_k}{2}$ for some $i, j, k\in [m-1]$, which is at most $\frac{a_m}{2}$ since $a_i, a_j\leq\frac{a_m}{2}$ and $a_k>0$.
Thus the above rank assignment is a $\left(\theta_1, \theta_2,\dots, \theta_{2m-1} \right)$-representation of $nK_3$.
Therefore, $nK_3$ is a $(2m-1)$-threshold graph, that is $\Theta(nK_3)\leq2m-1$ as desired.
	
Suppose that $n>q_{m-1}$.
To prove that $\Theta(nK_3)\geq2m$, we suppose that $\Theta(nK_3)\leq2m-1$.
Let $r$ be a $\left(\theta_1, \theta_2,\dots, \theta_{2m-1} \right)$-representation of $nK_3$.
Then there are at most $m$ colors of edges in $nK_3$.
By Lemma~\ref{Max.of.n}$(i)$, there are at most $q_{m-1}-1$ triangles without color $m$.
By Lemma~\ref{nocolorK3}$(i)$, an edge of color $m$ appears in at most one triangle.
Thus $n\leq \left(q_{m-1}-1\right)+1$, a contradiction.
Therefore, $\Theta(nK_3)\geq2m$.
\end{proof}
By applying Theorem~\ref{Main theorem K3}$(i)$ together with Proposition~\ref{complement}, we can narrow down the possible values of $\Theta(K_{n\times3})$ to just two numbers.
%%%%%%%%%%%%%%%%%%%%%%%%%%%%%%%%%%%%%%%%%%%
%\begin{thm} \label{TK33...3}
%	Let $S_k=k+\binom{k}{3}+2$.
%	For $n\geq 2$, 
%	\begin{equation*}
%		\Theta(K_{n\times3})=
%		\left\{
%		\begin{array}{l}
%			2m \hspace{.8cm}\text{if} \hspace{.2cm} n=S_{m-1},
%			\\
%			2m+1 \hspace{.2cm}\text{if} \hspace{.2cm} S_{m-1} < n < S_m.
%		\end{array}
%		\right.
%	\end{equation*}	
%\end{thm}

\begin{proof} [Proof of Theorem~\ref{Main theorem K3}$(ii)$]
Let $m$ be such that $p_{m-1} \leq n < p_m$.
By Theorem~\ref{Main theorem K3}$(i)$,  $\Theta(nK_3)\in \left\lbrace 2m, 2m+1 \right\rbrace$, and hence $\Theta(K_{n\times3})\in \left\lbrace 2m, 2m+1 \right\rbrace$ by Proposition~\ref{complement}.
	
Suppose that $n=p_{m-1}$. 
To prove that $\Theta(K_{n\times3})\leq2m$, 
let $A=\left\{a_1, a_2, \dots, a_{m+1}\right\}\subset\mathbb{R}$ be a linearly independent set over $\mathbb{Q}$ such that 
$a_1<a_2<\dots<a_{m+1}$, 
$-\left|a_i\right| \geq \frac{a_1}{3}$ and $\left|a_i\right|\leq\frac{a_{m+1}}{3}$ for all $i\in[m]\setminus\left\lbrace 1 \right\rbrace$.
We pick an $A\setminus\left\lbrace a_1, a_{m+1} \right\rbrace$-assignment for the first $(m-1)+\binom{m-1}{3}$ parts in $K_{n\times3}$, and let the last two parts have nonedge rank sums $a_1, a_1, a_1$ and $a_{m+1}, a_{m+1}, a_{m+1}$. 
Note that this is an $A$-assignment of $K_{n\times3}$.
By Lemma~\ref{EnotNE}$(ii)$, the edge and nonedge rank sums do not coincide. 
Let $\theta_1$ be smaller than all rank sums.
We then separate the edge and nonedge rank sums by putting two thresholds around each interval of nonedge rank sums.
For $i=1, 2,\dots, m+1$, let $\theta_{2i}=a_i$ and 
$\theta_{2i+1}=a_i+\varepsilon$ where $\varepsilon$ is a sufficiently small positive real number. 
Thus the above rank assignment is a $\left(\theta_1, \theta_2,\dots, \theta_{2m+3} \right)$-representation of $K_{n\times3}$.
In fact, we will show that we do not need the thresholds $\theta_1, \theta_2$ and $\theta_{2m+3}$ by proving that no rank sum is smaller than $\theta_2$ or larger than $\theta_{2m+2}$.
It is sufficient to show that the rank of each vertex is at least $\frac{\theta_{2}}{2}=\frac{a_1}{2}$ and at most $\frac{\theta_{2m+2}}{2}=\frac{a_{m+1}}{2}$.
This is clear for the last two parts with nonedge rank sums  $a_1, a_1, a_1$ and $a_{m+1}, a_{m+1}, a_{m+1}$.
For the other parts, the rank of each vertex is of the form $\frac{a_i+a_j-a_k}{2}$ for some $i, j, k\in [m]\setminus\left\lbrace 1 \right\rbrace$, which is at least $\frac{a_1}{2}$ and at most $\frac{a_{m+1}}{2}$ since $\frac{a_1}{3} \leq a_i, a_j, -a_k\leq\frac{a_{m+1}}{3}$.
Thus the above rank assignment is a $\left(\theta_3, \theta_4,\dots, \theta_{2m+2} \right)$-representation of $K_{n\times3}$.
Therefore, $K_{n\times3}$ is a $2m$-threshold graph, that is $\Theta(K_{n\times3})\leq2m$ as desired.
			
Suppose that $n>p_{m-1}$.
To prove that $\Theta(K_{n\times3})\geq2m+1$, we suppose that $\Theta(K_{n\times3})\leq2m$.
Let $r$ be a $\left(\theta_1, \theta_2,\dots, \theta_{2m} \right)$-representation of $K_{n\times3}$.
Then there are at most $m+1$ colors of nonedges in $K_{n\times3}$.
By Lemma~\ref{Max.of.n}$(ii)$, there are at most $p_{m-1}-2$ parts without colors $1$ and $m+1$.
By Lemma~\ref{nocolorK3}$(ii)$ and~\ref{nocolorK3}$(iii)$, a nonedge of color $1$ appears in at most one part and a nonedge of color $m+1$ also appears in at most one part.
Therefore, $n\leq \left(p_{m-1}-2\right)+1+1$, a contradiction.
\end{proof}
%%%%%%%%%%%%%%%%%%%%%%%%%%%%%%%%%%%%%%%%%%%
%%%%%%%%%%%%%%%%%%%%%%%%%%%%%%%%%%%%%%%%%%%

\section{Threshold numbers of $K_{n\times4}$ and $nK_4$} \label{scetion-K4andKnx4}

In this section, we focus on the threshold numbers of $K_{n\times4}$ and $nK_4$. 
We will need Lemmas~\ref{NoTwoSameColor} and~\ref{KeyLowerBound} as well as five new lemmas. 
Lemma~\ref{assign} identifies all sets of edge rank sums that can appear in a $K_4$.
Lemmas~\ref{lemK4} and~\ref{nocolorK4} are for the lower bounds where the key idea is in Lemma~\ref{lemK4}. We apply Lemmas~\ref{NoTwoSameColor} and~\ref{KeyLowerBound} to prove Lemma~\ref{lemK4}, which provide the maximum number of $K_4$'s and parts in terms of the number of colors.
Lemma~\ref{nocolorK4} improves the lower bounds obtained from Lemma~\ref{lemK4}.
On the other hand, Lemma~\ref{EnotNEK4} which is a tool to prove the upper bounds utilizes Lemma~\ref{sumeven} in its proof.  

We start by assigning a color to each edge of $nK_4$ and each nonedge of $K_{n\times4}$ as follows.
In a $\left(\theta_1, \theta_2,\dots, \theta_k \right)$-representation of $nK_4$ where $\theta_1<\theta_2<\cdots< \theta_k$, we color an edge $uv$ with \emph{color $i$}, for $i\in \left\lbrace 1, 2,\dots, \left\lceil{\frac{k}{2}} \right\rceil\right\rbrace$, if $r_u + r_v\in\left[\theta_{2i-1},\theta_{2i}\right)$.
Similarly, in a $\left(\theta_1, \theta_2,\dots, \theta_k \right)$-representation of $K_{n\times4}$ where $\theta_1<\theta_2<\cdots< \theta_k$, we color a nonedge $xy$ with \emph{color $i$}, for $i\in \left\lbrace 1, 2,\dots, \left\lceil{\frac{k+1}{2}} \right\rceil\right\rbrace$, if $r_x + r_y\in\left[\theta_{2i-2},\theta_{2i-1}\right)$ where $\theta_0=-\infty$.

\begin{figure*}[t]
	\centering
	\begin{tikzpicture}
		[scale=.6,auto=center]
		\tikzset{enclosed/.style={draw, circle, inner sep=0pt, minimum size=.1cm, fill=black}}
		
		\node[enclosed] (1) at (0,3) {};
		\node[enclosed] (2) at (3,3) {};
		\node[enclosed] (3) at (3,0) {};
		\node[enclosed] (4) at (0,0) {};
		
		\draw (1) -- (2) node[midway, above] (edge1) {$a_1$};
		\draw (1) -- (3) node[near start, below] (edge2) {$a_3$};
		\draw (1) -- (4) node[midway, left] (edge3) {$a_2$};
		\draw (2) -- (3) node[midway, right] (edge4) {$b_2$};
		\draw (2) -- (4) node[near start, below] (edge5) {$b_3$};
		\draw (3) -- (4) node[midway, below] (edge6) {$b_1$};
	\end{tikzpicture}
	\caption{$K_4\left(a_1, b_1, a_2, b_2, a_3, b_3\right)$ \label{F1}}
\end{figure*}
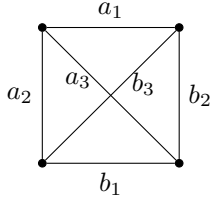

We denote by $K_4\left(a_1, b_1, a_2, b_2, a_3, b_3\right)$ a $K_4$ each of whose vertices is assigned a rank so that the edge rank sums are
$a_1$, $b_1$, $a_2$, $b_2$, $a_3$, $b_3$
where $a_i$ and $b_i$ belong to a perfect matching for each $i$ as shown in Figure~\ref{F1}.
Observe that $K_4\left(a_1, b_1, a_2, b_2, a_3, b_3\right)$ and $K_4\left(a_1, b_1, a_2, b_2, b_3, a_3\right)$ are not isomorphic, while $K_4\left(b_1, a_1, a_2, b_2, a_3, b_3\right)$, $K_4\left(a_1, b_1, b_2, a_2, a_3, b_3\right)$ and $K_4\left(a_1, b_1, a_2, b_2, b_3, a_3\right)$ are isomorphic.
For convenience, we write $K_4\left(c\right)$ for $K_4\left(c, c, c, c, c, c\right)$. 

In the same fasion, 
we denote by $E_4\left(a_1, b_1, a_2, b_2, a_3, b_3\right)$ an empty graph on four vertices having nonedge rank sums $a_1, b_1, a_2, b_2, a_3$ and $b_3$ where $a_i$ and $b_i$ belong to an independent nonedges for each $i$.

%%%%%%%%%%%%%%%%%%%%%%%%%%%%%%%%%%%%%%%%%%%
It is easy to determine which edge rank sums $a_1, b_1, a_2, b_2, a_3, b_3$ can appear in a $K_4$.
%%%%%%%%%%%%%%%%%%%%%%%%%%%%%%%%%%%%%%%%%%%
\begin{prop} \label{assign}
	The following statements are equivalent:
	\begin{enumerate}[(i)]	
		\item $K_4\left(a_1, b_1, a_2, b_2, a_3, b_3\right)$ exists.
		
		\item $E_4\left(a_1, b_1, a_2, b_2, a_3, b_3\right)$ exists.
		
		\item $a_1+b_1=a_2+b_2=a_3+b_3$.
	\end{enumerate}	
\end{prop}

\begin{proof}
	$(i)\Rightarrow(iii)$:	
	Suppose that $K_4\left(a_1, b_1, a_2, b_2, a_3, b_3\right)$ exists, that is
	we can assign a rank to each vertex so that the edge rank sums are
	$a_1$, $b_1$, $a_2$, $b_2$, $a_3$, $b_3$
	where $a_i$ and $b_i$ belong to a perfect matching for each $i$.
	Since each perfect matching spans all vertices of the graph, the summation of all ranks is equal to $a_i+b_i$ for each $i$.
	Thus $a_1+b_1=a_2+b_2=a_3+b_3$ as desired.
	\begin{figure*}[b]
		\centering
		\begin{tikzpicture}
			[scale=.6,auto=center]
			\tikzset{enclosed/.style={draw, circle, inner sep=0pt, minimum size=.1cm, fill=black}}
			
			\node[enclosed, label={north west: $w$}] (1) at (0,3) {};
			\node[enclosed, label={north east: $x$}] (2) at (3,3) {};
			\node[enclosed, label={south east: $z$}] (3) at (3,0) {};
			\node[enclosed, label={south west: $y$}] (4) at (0,0) {};
			
			\draw (1) -- (2) node[midway, above] (edge1) {$a_1$};
			\draw (1) -- (3) node[near start, below] (edge2) {$a_3$};
			\draw (1) -- (4) node[midway, left] (edge3) {$a_2$};
			\draw (2) -- (3) node[midway, right] (edge4) {$b_2$};
			\draw (2) -- (4) node[near start, below] (edge5) {$b_3$};
			\draw (3) -- (4) node[midway, below] (edge6) {$b_1$};
		\end{tikzpicture}
		\caption{\label{F3}}
	\end{figure*}	
	$(iii)\Rightarrow(i)$:
	Let $\left\lbrace w, x, y, z\right\rbrace$ be the vertex set of $K_4$. 
	We will provide an assignment $r$ of ranks so that the graph is $K_4\left(a_1, b_1, a_2, b_2, a_3, b_3\right)$ as shown in Figure~\ref{F3}.
	To obtain edge rank sums $b_1, b_2, b_3$ for the triangle $xyz$, we let
	\begin{align*}
		r(x)=\frac{b_2+b_3-b_1}{2},\ r(y)=\frac{b_1+b_3-b_2}{2},\ 
		r(z)=\frac{b_1+b_2-b_3}{2}.
	\end{align*}
	We immediately obtain 
	$r(y)+r(z)=b_1$,
	$r(x)+r(z)=b_2$ and 
	$r(x)+r(y)=b_3$.
	Now, let $r(w)=\frac{a_1+a_2-b_3}{2}$.
	Thus
	\begin{equation*}
		\begin{aligned}
			&r(w)+r(x)  = \frac{a_1+a_2-b_3}{2}+\frac{b_2+b_3-b_1}{2}=a_1\ \ \text{since $a_1+b_1=a_2+b_2$},\\
			&r(w)+r(y)  = \frac{a_1+a_2-b_3}{2}+\frac{b_1+b_3-b_2}{2}=a_2\ \ \text{since $a_1+b_1=a_2+b_2$},\\
			&r(w)+r(z)  = \frac{a_1+a_2-b_3}{2}+\frac{b_1+b_2-b_3}{2}=a_3\ \ \text{since $a_1+b_1=a_2+b_2=a_3+b_3$}.
		\end{aligned}
	\end{equation*} 
		
	For $(ii)\Leftrightarrow(iii)$, the proof is similar.		
\end{proof}
%%%%%%%%%%%%%%%%%%%%%%%%%%%%%%%%%%%%%%%%%%%
The following key lemma for the lower bounds for the threshold number, determines the maximum numbers of $K_4$'s and parts in terms of the number of colors.  
The crux of the proof is an observation that each $K_4$ must contain a particular kind of $K_3$.
%%%%%%%%%%%%%%%%%%%%%%%%%%%%%%%%%%%%%%%%%%%
\begin{lem} \label{lemK4}
	\begin{enumerate}[(i)]	
		\item If there are at most $m$ colors of edges in $nK_4$, then $n\leq  m+
		\binom{\left\lfloor m/2\right\rfloor}{3}+
		\binom{\left\lceil m/2\right\rceil}{3}$.
		In particular, if $nK_4$ is a $k$-threshold graph, then $n\leq  \left\lceil{\frac{k}{2}}\right\rceil+
		\binom{\left\lfloor (k+1)/4\right\rfloor}{3}+
		\binom{\left\lceil k/4\right\rceil}{3}$.
		\item If there are at most $m$ colors of nonedges in $K_{n\times4}$, then $n\leq  m+
		\binom{\left\lfloor m/2\right\rfloor}{3}+
		\binom{\left\lceil m/2\right\rceil}{3}$.	
		In particular, if $K_{n\times4}$ is a $k$-threshold graph, then $n\leq  \left\lceil{\frac{k+1}{2}}\right\rceil+
		\binom{\left\lfloor(k+2)/4\right\rfloor}{3}+
		\binom{\left\lceil(k+1)/4\right\rceil}{3}$.	
	\end{enumerate}
\end{lem} 

\begin{proof}
	We will only prove $(i)$ as the proof of $(ii)$ is similar.
	Let $r$ be a representation of $nK_4$ such that there are at most $m$ colors of edges.
	We decompose $nK_4$ into two subgraphs $G_1=n_1K_4$ and $G_2=n_2K_4$ with $n=n_1+n_2$ such that $G_1$ consists of all $K_4$'s containing a triangle whose edges are colored with $1$ or $2$ colors and $G_2$ consists of all $K_4$'s with four triangles whose edges are colored with $3$ colors.
	First, we show that $n_1\leq m$.  
	Consider a subgraph $n_1K_3$ of $G_1$ consisting of triangles whose edges are colored with $1$ or $2$ colors.
	Applying Lemma~\ref{KeyLowerBound} with the representation $r$ of $n_1K_3$, for each color $i\in[m]$, there is at most one triangle in $n_1K_3$ of color of the form $i j j$ where $j\in [m]$.
	Thus $n_1\leq m$.

	It remains to show that $n_2\leq\binom{\left\lfloor m/2\right\rfloor}{3}+
	\binom{\left\lceil m/2\right\rceil}{3}$. 
	Let $\mathcal{L}$ be the set of triangles in $nK_4$ of colors $ij\ell$ where $i, j, \ell\in \left\lbrace 1, 2,\dots, \left\lfloor\frac{m}{2}\right\rfloor\right\rbrace$ are all distinct and 
	let $\mathcal{U}$ be the set of triangles in $nK_4$ of colors $ij\ell$ where $i, j, \ell\in \left\lbrace \left\lfloor\frac{m}{2}\right\rfloor+1, \left\lfloor\frac{m}{2}\right\rfloor+2,\dots, m \right\rbrace$ are all distinct.
	
	\begin{claim}
		Each $K_4$ in $G_2$ contains at least one triangle in $\mathcal{L}\cup \mathcal{U}$.
	\end{claim}
	\begin{proof}[Proof of Claim]
		Let $\left\lbrace v_1, v_2, v_3, v_4 \right\rbrace$ be the vertex set of $K_4$.
		Suppose without loss of generality that $r_{v_1}\leq r_{v_2}\leq r_{v_3}\leq r_{v_4}$.
		Then 
		\begin{equation*}
			r_{v_1}+r_{v_2}\leq r_{v_1}+r_{v_3}\leq 
			r_{v_2}+r_{v_3}\leq
			r_{v_2}+r_{v_4}\leq 
			r_{v_3}+r_{v_4}.
		\end{equation*}
		Let $v_2v_3$ have color $i$.
		Thus $i$ is in either
		$\left\lbrace 1, 2,\dots, \left\lfloor\frac{m}{2}\right\rfloor\right\rbrace$ or 
		$\left\lbrace \left\lfloor\frac{m}{2}\right\rfloor+1, \left\lfloor\frac{m}{2}\right\rfloor+2,\dots, m \right\rbrace$, and hence either $v_1v_2v_3$ is in $\mathcal{L}$ or $v_2v_3v_4$ is in $\mathcal{U}$.
	\end{proof}
	
	Consider a subgraph $n_2K_3$ of $G_2$ consisting of triangles in $\mathcal{L}\cup\mathcal{U}$ which exists by Claim.
	Applying Lemma~\ref{NoTwoSameColor} with the representation $r$ of $n_2K_3$, no two triangles in $n_2K_3$ have the same color.
	Thus 
	\begin{equation*}
		n_2\leq \left|\mathcal{L}\cup \mathcal{U}\right|\leq \binom{\left\lfloor \frac{m}{2} \right\rfloor}{3}+\binom{\left\lceil \frac{m}{2} \right\rceil}{3}.
	\end{equation*}
	
	Observe that if $nK_4$ is a $k$-threshold graph, then there are at most $\left\lceil\frac{k}{2}\right\rceil$ colors of edges in $nK_4$, and
	hence
	\begin{equation*}
		n\leq 
		\left\lceil{\frac{k}{2}}\right\rceil+
		\binom{\left\lfloor \frac{\left\lceil k/2\right\rceil}{2} \right\rfloor}{3}+\binom{\left\lceil \frac{\left\lceil k/2\right\rceil}{2}\right\rceil}{3}
		=
		\left\lceil{\frac{k}{2}}\right\rceil+
		\binom{\left\lfloor \frac{k+1}{4}\right\rfloor}{3}+
		\binom{\left\lceil \frac{k}{4}\right\rceil}{3}.\qedhere
	\end{equation*}
\end{proof}
%%%%%%%%%%%%%%%%%%%%%%%%%%%%%%%%%%%%%%%%%%%
Similarly to the case of $K_3$, the lower bounds for the threshold numbers obtained from Lemma~\ref{lemK4} are not sharp. We again need another observation which says roughly that the first and last colors appear in at most
one $K_4$ or part.
%%%%%%%%%%%%%%%%%%%%%%%%%%%%%%%%%%%%%%%%%%%
\begin{lem} \label{nocolorK4} 
	\begin{enumerate}[(i)]
		\item In a $\left(\theta_1, \theta_2,\dots, \theta_{2m-1}\right)$-representation of $nK_4$, an edge of color $m$ appears in at most one $K_4$.
		\item In a $\left(\theta_1, \theta_2,\dots, \theta_m\right)$-representation of $K_{n\times4}$, a nonedge of color $1$ appears in at most one part.
		\item In a $\left(\theta_1, \theta_2,\dots, \theta_{2m}\right)$-representation of $K_{n\times4}$, a nonedge of color $m+1$ appears in at most one part.
	\end{enumerate}
\end{lem}

\begin{proof}
	The proof is similar to that of Lemma~\ref{nocolorK3}.
\end{proof}
%%%%%%%%%%%%%%%%%%%%%%%%%%%%%%%%%%%%%%%%%%%

The upper bounds for the threshold numbers will be obtained from rank assignments of the following forms.
Let $A=\left\lbrace a_1, a_2, \dots, a_M \right\rbrace$ and 
$B=\left\lbrace b_1, b_2, \dots, b_M \right\rbrace$
be such that $a_i+b_i=N$ for all $i\in[M]$.
For $n=2M+2\binom{M}{3}$,
the \emph{$\left(A, B\right)$-assignment} is the rank assignment of $nK_4$ consisting of the following $K_4$'s: 
\begin{align*}
	&K_4\left(a_i\right) \ \text{for each} \  i\in[M],\\
	&K_4\left(b_i\right) \ \text{for each} \  i\in[M],\\
	&K_4\left(a_i, b_i, a_j, b_j, a_k, b_k\right) \ \text{for each subset}  \left\lbrace i, j, k\right\rbrace \subset[M] \ \text{of size} \ 3, \\
	&K_4\left(a_i, b_i, a_j, b_j, b_k, a_k\right) \ \text{for each subset}  \left\lbrace i, j, k\right\rbrace \subset[M] \ \text{of size} \ 3
\end{align*}
where each of them appears exactly once.
Note that the numbers of $K_4$'s in each line are $M$, $M$, $\binom{M}{3}$, $\binom{M}{3}$ respectively, and they exist by Proposition~\ref{assign}.

Let $\varepsilon>0$. 
For $n=2M+1+\binom{M}{3}+\binom{M+1}{3}$,
the \emph{$\left(A, B, \varepsilon\right)$-assignment} is the rank assignment of $nK_4$ consisting of the following $K_4$'s:
\begin{align*}
	&K_4\left(a_i\right)\ \text{for each} \  i\in[M],\\
	&K_4\left(b_i\right)\ \text{for each} \  i\in[M],\\
	&K_4\left(a_i, b_i, a_j, b_j, a_k, b_k\right) \ \text{for each subset}  \left\lbrace i, j, k\right\rbrace \subset[M] \ \text{of size} \ 3,\\
	&K_4\left(a_i, b_i, a_j, b_j, b_k, a_k\right) \ \text{for each subset}  \left\lbrace i, j, k\right\rbrace \subset[M] \ \text{of size} \ 3,\\
	&K_4\left(\frac{N}{2}+\varepsilon\right),\\ 
	&K_4\left(a_i+\varepsilon, b_i+\varepsilon, a_j+\varepsilon, b_j+\varepsilon, \frac{N}{2}+\varepsilon, \frac{N}{2}+\varepsilon\right) \ \text{for each subset}  \left\lbrace i, j\right\rbrace \subset[M] \ \text{of size} \ 2
\end{align*}
where each of them appears exactly once.
Note that the numbers of $K_4$'s in each line are $M$, $M$, $\binom{M}{3}$, $\binom{M}{3}$, $1$, $\binom{M}{2}$ respectively, and they exist by Proposition~\ref{assign}. 

Occasionally, we say that a $K_4$
is of \emph{type I} if it is a $K_4\left(a_i\right)$ or $K_4\left(b_i\right)$ for some $i\in [M]$,
\emph{type II} if it is a $K_4\left(a_i, b_i, a_j, b_j, a_k, b_k\right)$ or $K_4\left(a_i, b_i, a_j, b_j, b_k, a_k\right)$ for some subset $\left\lbrace i, j, k\right\rbrace \subset [M]$ of size  $3$, 
\emph{type III} if it is a $K_4\left(\frac{N}{2}+\varepsilon\right)$,
\emph{type IV} if it is a $K_4\left(a_i+\varepsilon, b_i+\varepsilon, a_j+\varepsilon, b_j+\varepsilon, \frac{N}{2}+\varepsilon, \frac{N}{2}+\varepsilon\right)$  for some distinct $i, j \in [M]$.
In the same fasion, we can define the \emph{$\left(A, B\right)$-assignment} and the \emph{$\left(A, B, \varepsilon\right)$-assignment} of $K_{n\times4}$ by replacing $K_4$ with $E_4$. 

%\begin{figure*}[b]
%	\centering
%	\begin{tikzpicture} 
%		[scale=.6,auto=center]
%		\tikzset{enclosed/.style={draw, circle, inner sep=0pt, minimum size=.1cm, fill=black}}
%		
%		\node[enclosed] (1) at (0,6) {};
%		\node[enclosed] (2) at (0,4) {};
%		\node[enclosed] (3) at (0,2) {};
%		\node[enclosed] (4) at (0,0) {};
%		
%		\draw[dashed] (1) -- (2) node[midway, left] (edge1) {$a_1$};
%		\draw[dashed] (1) .. controls (0.8,4) .. (3) node[midway, right] (edge2) {$a_3$};
%		\draw[dashed] (1) .. controls (2.5,3) .. (4) node[midway, right] (edge3) {$a_2$};
%		\draw[dashed] (2) -- (3) node[midway, left] (edge4) {$b_2$};
%		\draw[dashed] (2) .. controls (0.8,2) .. (4) node[midway, right] (edge5) {$b_3$};
%		\draw[dashed] (3) -- (4) node[midway, left] (edge6) {$b_1$};
%	\end{tikzpicture}
%	\caption{$E_4\left(a_1, b_1, a_2, b_2, a_3, b_3\right)$ \label{F2}}
%\end{figure*}

%%%%%%%%%%%%%%%%%%%%%%%%%%%%%%%%%%%%%%%%%%%
The following lemma will be used repeatedly in the proof of Lemma~\ref{EnotNEK4}.
%%%%%%%%%%%%%%%%%%%%%%%%%%%%%%%%%%%%%%%%%%%
\begin{lem} \label{sumeven}
	Let $\left\lbrace N, a_1, a_2, \dots, a_M \right\rbrace\subset\mathbb{R}$ be a linearly independent set over $\mathbb{Q}$ and 
	$b_i=N-a_i$ for $i=1, 2, \dots, M$. 
	Let $A=\left\lbrace a_1, a_2, \dots, a_M \right\rbrace$ and 
	$B=\left\lbrace b_1, b_2, \dots, b_M \right\rbrace$.
	If 
	\begin{equation*}
	\sum_{i=1}^{S}\alpha_ix_i + \beta N = 0 
	\end{equation*}
where $\alpha_i \in \mathbb{Z}$, 
$x_i\in A \cup B$ for all $i\in [S]$ and $\beta \in \mathbb{Q}$, then
$\sum_{i=1}^{S}\alpha_i$ is even.
\end{lem} 

\begin{proof}
Suppose that 
$\sum_{i=1}^{S}\alpha_ix_i + \beta N = 0 $ where 
$\alpha_i \in \mathbb{Z}$, 
$x_i\in A \cup B$ for all $i\in [S]$ and $\beta \in \mathbb{Q}$.
Observe that $x_i$ is either $a_{j_i}$ or $b_{j_i}=N-a_{j_i}$ where $j_i\in [M]$.
Then we can write 
$x_i=\delta_ia_{j_i}+ \beta_i N$
where 
$\delta_i\in \left\lbrace -1, 1 \right\rbrace$ and $\beta_i\in \left\lbrace 0, 1 \right\rbrace$.
The equation becomes
\begin{equation*}
	\sum_{i=1}^{S}\delta_i\alpha_ia_{j_i} + \sum_{i=1}^{S}\beta_i\alpha_i N + \beta N = 0.
\end{equation*}
Since $\left\lbrace N, a_1, a_2, \dots, a_M \right\rbrace$ is linearly independent over $\mathbb{Q}$, 
we have $\sum_{i=1}^{S}\delta_i\alpha_i=0$.
Hence, 
\begin{equation*}
	\sum_{i=1}^{S}\alpha_i = \sum_{i=1}^{S}\delta_i\alpha_i +
	2\sum_{\delta_i=-1}\alpha_i = 2\sum_{\delta_i=-1}\alpha_i
\end{equation*}
is even.
\end{proof}
%%%%%%%%%%%%%%%%%%%%%%%%%%%%%%%%%%%%%%%%%%%
The linear independence of $\left\lbrace N, a_1, a_2, \dots, a_M \right\rbrace$ over $\mathbb{Q}$ is a sufficient condition for the edge and nonedge rank sums in the $\left(A, B\right)$-assignment and in the $\left(A,  B, \varepsilon\right)$-assignment not to coincide.
For the $\left(A,  B, \varepsilon\right)$-assignment, we prove further that there are small intervals without nonedge rank sums that cover all edge rank sums.
%%%%%%%%%%%%%%%%%%%%%%%%%%%%%%%%%%%%%%%%%%%
\begin{lem} \label{EnotNEK4}
Let $\left\lbrace N, a_1, a_2, \dots, a_M \right\rbrace\subset\mathbb{R}$ be a linearly independent set over $\mathbb{Q}$ and $b_i=N-a_i$ for $i=1, 2, \dots, M$. 
Let $A=\left\lbrace a_1, a_2, \dots, a_M \right\rbrace$ and $B=\left\lbrace b_1, b_2, \dots, b_M\right\rbrace$.
\begin{enumerate}[(i)]
\item Let $n=2M+2\binom{M}{3}$. In the $\left(A, B\right)$-assignment of $nK_4$, the edge and nonedge rank sums do not coincide.

\item Let $n=2M+1+\binom{M}{3}+\binom{M+1}{3}$. 
Then there exists a positive real number  $\varepsilon$ such that, in the $\left(A,  B, \varepsilon\right)$-assignment of $nK_4$, no nonedge rank sum lies in either  $\left[a_i, a_i+\varepsilon\right]$, 
$\left[b_i, b_i+\varepsilon \right]$
or $\left\lbrace \frac{N}{2}+\varepsilon \right\rbrace$ for all $i\in[M]$.
Moreover, the sets of the form $\left[a_i, a_i+\varepsilon\right]$, 
$\left[b_i, b_i+\varepsilon \right]$
and $\left\lbrace \frac{N}{2}+\varepsilon \right\rbrace$ for all $i\in[M]$ are pairwise disjoint.

\item Let $n=2M+2\binom{M}{3}$. In the $\left(A, B\right)$-assignment of $K_{n\times4}$, the edge and nonedge rank sums do not coincide.

\item Let $n=2M+1+\binom{M}{3}+\binom{M+1}{3}$. 
Then there exists a positive real number  $\varepsilon$ such that, in the $\left(A, B, \varepsilon\right)$-assignment of $K_{n\times4}$, no edge rank sum lies in either  $\left[a_i, a_i+\varepsilon\right]$, 
$\left[b_i, b_i+\varepsilon \right]$
or $\left\lbrace \frac{N}{2}+\varepsilon \right\rbrace$ for all $i\in[M]$.
Moreover, the sets of the form $\left[a_i, a_i+\varepsilon\right]$, 
$\left[b_i, b_i+\varepsilon \right]$
and $\left\lbrace \frac{N}{2}+\varepsilon \right\rbrace$ for all $i\in[M]$ are pairwise disjoint.
\end{enumerate}
\end{lem} 

\begin{proof}
For $(i)$ and $(ii)$, it is sufficient to prove $(ii)$ since every $K_4$ in the $\left(A, B\right)$-assignment appears in the $\left(A, B, \varepsilon\right)$-assignment and each edge rank sum in the $\left(A, B\right)$-assignment is either $a_i$ or $b_i$.
The proofs of $(iii)$ and $(iv)$ are similar to those of $(i)$ and $(ii)$.

To prove $(ii)$, let  $n=2M+1+\binom{M}{3}+\binom{M+1}{3}$.
We first consider the $\left(A, B, \varepsilon\right)$-assignment of $nK_4$ in the case when $\varepsilon=0$.

\begin{claim}
For the $\left(A, B, 0\right)$-assignment $r'$ of $nK_4$, no nonedge rank sum lies in $A\cup B$.
\end{claim}

\begin{proof}[Proof of Claim]
\begin{table}[b]
	\centering
	\begin{tabular}{|c||c|c|c|c|}
		\hline
		\diagbox[width=4em]{$x$}{$y$} & I $K_4$ & II $K_4$ & III $K_4$ & IV $K_4$\\
		\hline
		\hline
		I $K_4$ & Case 1 & Case 2 & Case 4 & Case 5 \\ 
		\hline
		II $K_4$ & \cellcolor{lightgray} & Case 3 & Case 6 & Case 7\\ 
		\hline
		III $K_4$ & \cellcolor{lightgray} & \cellcolor{lightgray} & Case 8 & Case 9\\ 
		\hline
		IV $K_4$ & \cellcolor{lightgray} & \cellcolor{lightgray} & \cellcolor{lightgray} & Case 10\\
		\hline
	\end{tabular}
	\caption{Ten cases according to the four possible types of $K_4$ that $x$ and $y$ are in.}
	\label{tabcase}
\end{table}	
Suppose to the contrary that there exists a nonedge $xy$ in $nK_4$ such that $r'_x+r'_y$ lies in $A\cup B$, 
say $r'_x+r'_y=e_t\in \left\lbrace a_t, b_t\right\rbrace$ for some $t\in[M]$.	
We divide into cases according to the four possible types of $K_4$ that $x$ and $y$ are in as shown in Table~\ref{tabcase}.

Observe that the rank of each vertex in a type I $K_4$  is of the form $\frac{c_i}{2}$ 
where $i\in[M]$ and $c_i\in \left\lbrace a_i, b_i \right\rbrace$, 
that in a type II $K_4$ is of the form $\frac{c_i+c_j-c_k}{2}$ where 
$ i, j, k\in[M]$ are all distinct and
$c_\ell\in \left\lbrace a_\ell, b_\ell \right\rbrace$ for  $\ell \in \left\lbrace i, j, k\right\rbrace$, 
that in a type III $K_4$ is of the form $\frac{N}{4}$,
and 
that in a type IV $K_4$ is of the form $\frac{c_i+c_j-N/2}{2}$ where 
$ i, j\in[M]$ are distinct and
$c_\ell\in \left\lbrace a_\ell, b_\ell \right\rbrace$ for  $\ell \in \left\lbrace i, j\right\rbrace$.

\noindent
\textbf{Case 1.} 
$x, y\in\text{type I} \ K_4$. 

Then $x\in K_4\left(c_i\right)$ and $y\in K_4\left(d_j\right)$ 
where  
$i, j\in[M]$ and
$c_i\in \left\lbrace a_i, b_i\right\rbrace$, 
$d_j\in \left\lbrace a_j, b_j\right\rbrace$.
Thus $r'_x=\frac{c_i}{2}$ and $r'_y=\frac{d_j}{2}$.
The equation $r'_x+r'_y=e_t$ becomes 
\begin{equation*}
	c_i+d_j=2e_t.
\end{equation*}
First, suppose that $i\neq j$.
One of $i$ or $j$ cannot equal to $t$, say $i\neq t$.
By writing the equation in terms of the basis $\left\lbrace N, a_1, a_2, \dots, a_M \right\rbrace$, we can see that
the equality cannot occur since $c_i$ is the only term in the equation involving $a_i$, a contradiction.
Now, suppose that $i=j$.
Since $x$ and $y$ are in different $K_4$'s, we have 
$\left\lbrace c_i, d_j \right\rbrace =\left\lbrace a_i, b_i \right\rbrace$, and hence the equation becomes $N=2e_t$, a contradiction.

\noindent
\textbf{Case 2.} 
$x\in\text{type I} \ K_4$ and $y\in\text{type II} \ K_4$.

Then $x\in K_4\left(c_i\right)$ where
$i\in[M]$ and $c_i\in \left\lbrace a_i, b_i \right\rbrace$, and 
$y$ is in either $K_4\left(a_p, b_p, a_q, b_q, a_s, b_s\right)$ or $K_4\left(a_p, b_p, a_q, b_q, b_s, a_s\right)$
where $p, q, s \in [M]$ are all distinct.
Thus
$r'_x=\frac{c_i}{2}$ and $r'_y=\frac{d_p+d_q-d_s}{2}$ where 
$d_\ell\in \left\lbrace a_\ell, b_\ell\right\rbrace$ for $\ell\in\left\lbrace p, q, s  \right\rbrace$.
The equation $r'_x+r'_y=e_t$ becomes 
\begin{equation*}
	c_i+d_p+d_q-d_s=2e_t.
\end{equation*}
Since $p, q, s$ are all distinct,
there is an index in $\left\lbrace p, q, s\right\rbrace$ not appearing in $\left\lbrace i, t\right\rbrace$, say $p\notin \left\lbrace i, t\right\rbrace$.
Thus the equality cannot occur since $d_p$ is the only term in the equation involving $a_p$, a contradiction.

\noindent
\textbf{Case 3.} 
$x, y\in\text{type II} \ K_4$.

Then $x$ is in either $K_4\left(a_i, b_i, a_j, b_j, a_k, b_k\right)$ or $K_4\left(a_i, b_i, a_j, b_j, b_k, a_k\right)$
where $i, j, k \in [M]$ are all distinct,
and 
$y$ is in either $K_4\left(a_p, b_p, a_q, b_q, a_s, b_s\right)$ or $K_4\left(a_p, b_p, a_q, b_q, b_s, a_s\right)$
where $p, q, s \in [M]$ are all distinct.
Thus
$r'_x=\frac{c_i+c_j-c_k}{2}$ where 
$c_\ell\in \left\lbrace a_\ell, b_\ell\right\rbrace$ for $\ell\in\left\lbrace i, j, k \right\rbrace$ and 
$r'_y=\frac{d_p+d_q-d_s}{2}$ where 
$d_\ell\in \left\lbrace a_\ell, b_\ell\right\rbrace$ for $\ell\in\left\lbrace p, q, s \right\rbrace$.
The equation $r'_x+r'_y=e_t$ becomes 
\begin{equation*}
	c_i+c_j-c_k+d_p+d_q-d_s=2e_t.
\end{equation*}

\noindent
\textbf{Case 3.1.} 
$\left\lbrace i, j, k \right\rbrace\neq\left\lbrace p, q, s \right\rbrace$. 

Then there is an index in $\left\lbrace i, j, k \right\rbrace$ not appearing in $\left\lbrace p, q, s \right\rbrace$, say 
$i\notin \left\lbrace p, q, s \right\rbrace$.
Similarly, there is an index in $\left\lbrace p, q, s \right\rbrace$ not appearing in $\left\lbrace i, j, k \right\rbrace$, say $p\notin\left\lbrace i, j, k \right\rbrace$.
One of $i$ or $p$ cannot equal to $t$, say $i\neq t$. 
Thus the equality cannot occur since $c_i$ is the only term in the equation involving $a_i$, a contradiction.

\begin{table}[t]
	\begin{subtable}[c]{0.5\textwidth}
		\centering
		\begin{tabular}{|c|c|c|}
			\hline
			$c_i$ & $c_j$ & $c_k$ \\ 
			\hline
			\hline
			$a_i$ & $b_j$ & $a_k$ \\ 
			\hline
			$b_i$ & $b_j$ & $b_k$ \\ 
			\hline
			$b_i$ & $a_j$ & $a_k$ \\ 
			\hline
			$a_i$ & $a_j$ & $b_k$ \\
			\hline
		\end{tabular}
		\subcaption{$K_4\left(a_i, b_i, a_j, b_j, a_k, b_k\right)$}
		\label{tab1}
	\end{subtable}
	\begin{subtable}[c]{0.5\textwidth}
		\centering
		\begin{tabular}{|c|c|c|}
			\hline
			$d_i$ & $d_j$ & $d_k$ \\ 
			\hline
			\hline
			$a_i$ & $b_j$ & $b_k$ \\ 
			\hline
			$b_i$ & $b_j$ & $a_k$ \\ 
			\hline
			$b_i$ & $a_j$ & $b_k$ \\ 
			\hline
			$a_i$ & $a_j$ & $a_k$ \\
			\hline
		\end{tabular}
		\subcaption{$K_4\left(a_i, b_i, a_j, b_j, b_k, a_k\right)$}
		\label{tab2}
	\end{subtable}
	\caption{The possible values of $c_i, c_j, c_k$ and $d_i, d_j, d_k$.}
\end{table}

\noindent
\textbf{Case 3.2.} 
$\left\lbrace i, j, k \right\rbrace=\left\lbrace p, q, s \right\rbrace$.

Without loss of generality, let $i=p$, $j=q$ and $k=s$.
Since $x$, $y$ are in different $K_4$'s, we can assume without loss of generality that 
$x\in K_4\left(a_i, b_i, a_j, b_j, a_k, b_k\right)$ and 
$y\in K_4\left(a_i, b_i, a_j, b_j, b_k, a_k\right)$.
By considering the edge rank sum of each triangle in $K_4\left(a_i, b_i, a_j, b_j, a_k, b_k\right)$, 
each row in Table~\ref{tab1} shows the possible values of $c_i, c_j, c_k$, 
and 
by considering the edge rank sum of each triangle in $K_4\left(a_i, b_i, a_j, b_j, b_k, a_k\right)$,
each row in Table~\ref{tab2} shows the possible values of $d_i, d_j, d_k$.

By comparing a row in Table~\ref{tab1} with a row in Table~\ref{tab2}, we observe that either none or two of $c_i=d_i, c_j=d_j, c_k=d_k$ hold.
If none holds, then 
$\left\lbrace c_i, d_i \right\rbrace=\left\lbrace a_i, b_i \right\rbrace, \left\lbrace c_j, d_j \right\rbrace=\left\lbrace a_j, b_j \right\rbrace$ and $\left\lbrace c_k, d_k \right\rbrace=\left\lbrace a_k, b_k \right\rbrace$.
Thus the above equation becomes
\begin{equation*}
	N+N-N=2e_t
\end{equation*}
which is a contradiction.
If two of $c_i=d_i, c_j=d_j, c_k=d_k$ hold, then
we assume without loss of generality that $\left\lbrace c_i, d_i \right\rbrace=\left\lbrace a_i, b_i \right\rbrace$ and
$c_j=d_j$,  $c_k=d_k$.
Thus the original equation becomes
\begin{equation*}
	N+2c_j-2c_k=2e_t.
\end{equation*}
Since $j\neq k$, one of $j$ or $k$ cannot equal to $t$, say $j\neq t$.
Hence, the equality cannot occur since $c_j$ is the only term in the equation involving $a_j$, a contradiction.

\noindent
\textbf{Case 4.} 
$x\in\text{type I} \ K_4$ and $y\in\text{type III} \ K_4$.

Then $x\in K_4\left(c_i\right)$ where
$i\in[M]$ and $c_i\in \left\lbrace a_i, b_i\right\rbrace$, and 
$y\in K_4\left(\frac{N}{2}\right)$.
Thus
$r'_x=\frac{c_i}{2}$ and $r'_y=\frac{N}{4}$.
The equation $r'_x+r'_y=e_t$ becomes 
\begin{align*}
	c_i-2e_t +\frac{N}{2}& = 0
\end{align*}
By Lemma~\ref{sumeven}, the sum of the coefficients of $c_i$ and $e_t$ must be even, a contradiction.

\noindent
\textbf{Case 5.} 
$x\in\text{type I} \ K_4$ and $y\in\text{type IV} \ K_4$.

Then $x\in K_4\left(c_i\right)$ where
$i\in[M]$ and $c_i\in \left\lbrace a_i, b_i\right\rbrace$, and 
$y\in K_4\left(a_p, b_p, a_q, b_q, \frac{N}{2}, \frac{N}{2}\right)$
where  $p, q \in[M]$ are distinct.
Thus $r'_x=\frac{c_i}{2}$ and $r'_y=\frac{d_p+d_q-N/2}{2}$ where $d_\ell\in \left\lbrace a_\ell, b_\ell\right\rbrace$ for $\ell\in\left\lbrace p, s \right\rbrace$.
The equation $r'_x+r'_y=e_t$ becomes  
\begin{align*}
	c_i+d_p+d_q-2e_t-\frac{N}{2} = 0 
\end{align*}
By Lemma~\ref{sumeven},  
the sum of the coefficients of $c_i$, $d_p$, $d_q$ and $e_t$ must be even, a contradiction.

\noindent
\textbf{Case 6.}
$x\in\text{type II} \ K_4$ and $y\in\text{type III} \ K_4$.

Then $x$ is in either $K_4\left(a_i, b_i, a_j, b_j, a_k, b_k\right)$ or $K_4\left(a_i, b_i, a_j, b_j, b_k, a_k\right)$
where $ i, j, k \in[M]$ are all distinct, and
$y\in K_4\left(\frac{N}{2}\right)$.
Thus $r'_x=\frac{c_i+c_j-c_k}{2}$ where 
$c_\ell\in \left\lbrace a_\ell, b_\ell\right\rbrace$ for $\ell\in\left\lbrace i, j, k  \right\rbrace$, and 
$r'_y=\frac{N}{4}$.
The equation $r'_x+r'_y=e_t$ becomes 
\begin{align*}
	c_i+c_j-c_k+\frac{N}{2}-2e_t = 0
\end{align*}
By Lemma~\ref{sumeven},
the sum of the coefficients of $c_i$, $c_j$, $c_k$ and $e_t$ must be even, a contradiction.

\noindent
\textbf{Case 7.}
$x\in\text{type II} \ K_4$ and $y\in\text{type IV} \ K_4$.

Then $x$ is in either $K_4\left(a_i, b_i, a_j, b_j, a_k, b_k\right)$ or $K_4\left(a_i, b_i, a_j, b_j, b_k, a_k\right)$
where $ i, j, k \in[M]$ are all distinct, and
$y\in K_4\left(a_p, b_p, a_q, b_q, \frac{N}{2}, \frac{N}{2}\right)$
where $p, q \in[M]$ are distinct.
Thus $r'_x=\frac{c_i+c_j-c_k}{2}$ where 
$c_\ell\in \left\lbrace a_\ell, b_\ell\right\rbrace$ for $\ell\in\left\lbrace i, j, k  \right\rbrace$, and 
$r'_y=\frac{d_p+d_q-N/2}{2}$ where $d_\ell\in \left\lbrace a_\ell, b_\ell\right\rbrace$ for $\ell\in\left\lbrace p, q \right\rbrace$.
The equation $r'_x+r'_y=e_t$ becomes 
\begin{align*}
c_i+c_j-c_k+d_p+d_q-2e_t-\frac{N}{2} = 0
\end{align*}
By Lemma~\ref{sumeven}, 
the sum of the coefficients of $c_i$, $c_j$, $c_k$, $d_p$, $d_q$ and $e_t$ must be even, a contradiction.

\noindent
\textbf{Case 8.}
$x, y\in\text{type III} \ K_4$.

This case cannot occur
since $x$ and $y$ are in different $K_4$'s, but 
there is only one $K_4\left(\frac{N}{2}\right)$. 

\noindent
\textbf{Case 9.}
$x\in\text{type III} \ K_4$ and $y\in\text{type IV} \ K_4$.

Then $x\in K_4\left(\frac{N}{2}\right)$ and
$y\in K_4\left(a_i, b_i, a_j, b_j, \frac{N}{2}, \frac{N}{2}\right)$
where $i, j \in[M]$ are distinct.
Thus $r'_x=\frac{N}{4}$ and 
$r'_y=\frac{c_i+c_j-N/2}{2}$ where $c_\ell\in \left\lbrace a_\ell, b_\ell\right\rbrace$ for $\ell\in\left\lbrace i, j \right\rbrace$.
The equation $r'_x+r'_y=e_t$ becomes
\begin{align*}
	c_i+c_j=2e_t.
\end{align*}
We obtain a contradiction similar to Case 1.

\noindent
\textbf{Case 10.}
$x, y\in\text{type IV} \ K_4$.

Then 
$x\in K_4\left(a_i, b_i, a_j, b_j, \frac{N}{2}, \frac{N}{2}\right)$
where $i, j \in[M]$ are distinct, and 
$y\in K_4\left(a_p, b_p, a_q, b_q, \frac{N}{2}, \frac{N}{2}\right)$
where $p, q \in[M]$ are distinct.
Thus
$r'_x=\frac{c_i+c_j-N/2}{2}$ where $c_\ell\in \left\lbrace a_\ell, b_\ell\right\rbrace$ for $\ell\in\left\lbrace i, j \right\rbrace$, and
$r'_y=\frac{d_p+d_q-N/2}{2}$ where $d_\ell\in \left\lbrace a_\ell, b_\ell\right\rbrace$ for $\ell\in\left\lbrace p, q \right\rbrace$.
The equation $r'_x+r'_y=e_t$ becomes 
\begin{align*}
	c_i+c_j+d_p+d_q-N=2e_t.
\end{align*}
Since $x$ and $y$ are in different $K_4$'s, we have
$\left\lbrace i, j \right\rbrace \neq \left\lbrace p, q \right\rbrace$. 
Thus there is an index in one set not appearing in the other set, say 
$i\notin \left\lbrace p, q \right\rbrace$
and $p\notin\left\lbrace i, j \right\rbrace$.
One of $i$ or $p$ cannot equal to $t$, say $i\neq t$. 
Therefore, the equality cannot occur since $c_i$ is the only term in the equation involving $a_i$, a contradiction.
\end{proof}

Let $\varepsilon$ be a positive real number smaller than any distance between two distinct rank sums in the $\left(A, B, 0 \right)$-assignment of $nK_4$.
Note that the set of edge rank sums in the $\left(A, B, 0 \right)$-assignment of $nK_4$ is $A\cup B\cup \left\lbrace \frac{N}{2} \right\rbrace$.
By the definition of $\varepsilon$, the sets of the form $\left[a_i, a_i+\varepsilon\right]$, 
$\left[b_i, b_i+\varepsilon \right]$
and $\left\lbrace \frac{N}{2}+\varepsilon \right\rbrace$ for all $i\in[M]$ are pairwise disjoint.
Let $r$ be the $\left(A, B, \varepsilon \right)$-assignment of $nK_4$.
Then, for any vertex $u\in nK_4$,
\begin{equation*}
	r_u=
	\left\{
	\begin{array}{l}
		r'_u \hspace{.85cm}\emph{if} \hspace{.2cm} u \ \text{is in a type I or II}\ K_4,
		\\
		r'_u+\frac{\varepsilon}{2} \hspace{.2cm}\emph{if} \hspace{.2cm} u \ \text{is in a type III or IV}\ K_4.
	\end{array}
	\right.
\end{equation*}	
Let $xy$ be a nonedge in $nK_4$ and consider $a_i\in A$.
Observe that 
\begin{equation*}
	r_x+r_y\in\left\lbrace r'_x+r'_y,  r'_x+r'_y+\frac{\varepsilon}{2},  r'_x+r'_y+\varepsilon \right\rbrace.
\end{equation*}
\noindent
By Claim, $r'_x+r'_y\neq a_i$.
Since $a_i$ is a rank sum in the $\left(A, B, 0 \right)$-assignment, the distance between $r'_x+r'_y$ and $a_i$ exceeds $\varepsilon$ by the definition of $\varepsilon$.
If $r'_x+r'_y>a_i$, then $a_i+\varepsilon<r'_x+r'_y\leq r_x+r_y$.
If $r'_x+r'_y<a_i$, then $r_x+r_y\leq r'_x+r'_y+\varepsilon<a_i$.
Thus $r_x+r_y\notin\left[a_i, a_i+\varepsilon\right]$.
Similarly, 
$r_x+r_y\notin\left[b_i, b_i+\varepsilon \right]$.

It remains to show that $r_x+r_y\neq\frac{N}{2}+\varepsilon$.
Note that $\frac{N}{2}$ is a rank sum in the $\left(A, B, 0 \right)$-assignment.
Thus the distance between $r'_x+r'_y$ and $\frac{N}{2}$ is either $0$ or more than $\varepsilon$ by the definition of $\varepsilon$.
If $x$ or $y$ is in a type I or II $K_4$, then $r_x+r_y\in \left\lbrace r'_x+r'_y, r'_x+r'_y+\frac{\varepsilon}{2} \right\rbrace$.
If $r_x+r_y= \frac{N}{2}+\varepsilon$, then
the distance between $r'_x+r'_y$ and $\frac{N}{2}$ is either $\varepsilon$ or $\frac{\varepsilon}{2}$, a contradiction.
Thus we may suppose that both $x$ and $y$ are in a  a type III or IV $K_4$.
Since there is only one $K_4$ of type III, we may suppose further that $x$ is in a type IV $K_4$.
Then $x\in K_4\left(a_i+\varepsilon, b_i+\varepsilon, a_j+\varepsilon, b_j+\varepsilon, \frac{N}{2}+\varepsilon, \frac{N}{2}+\varepsilon\right)$
for some distinct $i, j \in[M]$.
Thus
$r_x=\frac{c_i+c_j-N/2+\varepsilon}{2}$ where $c_\ell\in \left\lbrace a_\ell, b_\ell\right\rbrace$ for $\ell\in\left\lbrace i, j \right\rbrace$.

If $y$ is in a type III $K_4$, then 
$y\in K_4\left(\frac{N}{2}+\varepsilon\right)$.
Thus $r_y=\frac{N}{4}+\frac{\varepsilon}{2}$.
Hence,
\begin{equation*}
	r_x+r_y=\left( \frac{c_i+c_j-N/2+\varepsilon}{2}\right)+\left( \frac{N}{4}+\frac{\varepsilon}{2}\right) =\frac{c_i+c_j}{2}+\varepsilon.
\end{equation*}
Suppose to the contrary that $r_x+r_y=\frac{N}{2}+\varepsilon$, that is $c_i+c_j=N$.
Since $i\neq j$, we have $c_i$ is the only term in the equation involving $a_i$.
Thus the equality cannot occur, a contradiction.

If $y$ is in a type IV $K_4$, then
$y\in K_4\left(a_p+\varepsilon, b_p+\varepsilon, a_q+\varepsilon, b_q+\varepsilon, \frac{N}{2}+\varepsilon, \frac{N}{2}+\varepsilon\right)$
for some distinct $p, q \in[M]$.
Thus
$r_y=\frac{d_p+d_q-N/2+\varepsilon}{2}$ where $d_\ell\in \left\lbrace a_\ell, b_\ell\right\rbrace$ for $\ell\in\left\lbrace p, q \right\rbrace$.
Suppose to the contrary that $r_x+r_y=\frac{N}{2}+\varepsilon$, i.e.
\begin{equation*}
c_i+c_j+d_p+d_q-2N=0.
\end{equation*}
Since $x$ and $y$ are in different $K_4$'s, we have
$\left\lbrace i, j \right\rbrace\neq\left\lbrace p, q \right\rbrace$. 
Thus there exists an index in one set not appearing in the other set, say $i\notin\left\lbrace p, q \right\rbrace$.
Recall that $i\neq j$.
Hence, the equality cannot occur since $c_i$ is the only term in the equation involving $a_i$, a contradiction.
\end{proof}
%%%%%%%%%%%%%%%%%%%%%%%%%%%%%%%%%%%%%%%%%%%
%%%%%%%%%%%%%%%%%%%%%%%%%%%%%%%%%%%%%%%%%%%
Now, we are ready to prove Theorem~\ref{Main theorem K4}.
Its proof follows the same line of argument as in the proof of Theorem~\ref{Main theorem K3}, nevertheless,
that of Theorem~\ref{Main theorem K4} is significantly more complicated.  
%%%%%%%%%%%%%%%%%%%%%%%%%%%%%%%%%%%%%%%%%%%
%\begin{thm} \label{TK4}
%Let $T_k=k+\binom{\left\lfloor{k/2}\right\rfloor}{3}+\binom{\left\lceil{k/2}\right\rceil}{3}+1$. Then
%\begin{equation*}
%	\Theta(nK_4)=
%	\left\{
%	\begin{array}{l}
%		2m-1 \hspace{.2cm}\text{if} \hspace{.2cm} n=T_{m-1},
%		\vspace{.25cm}\\
%		2m \hspace{.8cm}\text{if} \hspace{.2cm} 
%		T_{m-1}< n < T_m.
%	\end{array}
%	\right.
%\end{equation*}	
%\end{thm}

\begin{proof} [Proof of Theorem~\ref{Main theorem K4}$(i)$]
Let $m$ be such that $t_{m-1} \leq n < t_m$.
Suppose to the contrary that $\Theta(nK_4)\leq 2m-2$. 
By Lemma~\ref{lemK4}$(i)$, 
\begin{equation*}
	\begin{split}
		n & \leq \left\lceil{\frac{\Theta(nK_4)}{2}}\right\rceil+\binom{\left\lfloor\frac{\Theta(nK_4)+1}{4}\right\rfloor}{3}+\binom{\left\lceil\frac{\Theta(nK_4)}{4}\right\rceil}{3} \\
		& \leq \left\lceil{\frac{2m-2}{2}}\right\rceil+\binom{\left\lfloor\frac{2m-2+1}{4}\right\rfloor}{3}+\binom{\left\lceil\frac{2m-2}{4}\right\rceil}{3} \\
		& = (m-1)+\binom{\left\lfloor\frac{m-1}{2}\right\rfloor}{3}+\binom{\left\lceil\frac{m-1}{2}\right\rceil}{3} \\
		& =
		t_{m-1}-1
	\end{split}
\end{equation*}
contradicting the definition of $m$.
Hence, $\Theta(nK_4)\geq2m-1$. 

To prove that $\Theta(nK_4)\leq2m$, 
let $\left\lbrace N, a_1, a_2,\dots,a_{\left\lfloor{m/2}\right\rfloor}\right\rbrace\subset \mathbb{R^+}$ be a linearly independent set over $\mathbb{Q}$ such that $a_i<N$ for all $i\in \left\lbrace 1, 2,\dots, \left\lfloor{\frac{m}{2}}\right\rfloor\right\rbrace$
and 
let $b_i=N-a_i$ for $i= 1, 2,\dots, \left\lfloor{\frac{m}{2}}\right\rfloor$. 
Write $A=\left\lbrace a_1, a_2,\dots, a_{\left\lfloor{m/2}\right\rfloor}\right\rbrace$ and $B=\left\lbrace  b_1, b_2,\dots, b_{\left\lfloor{m/2}\right\rfloor}\right\rbrace$.

\noindent
\textbf{Case 1.} $m$ is even. 	

Let $n'=t_{m}-1=m+2\binom{m/2}{3}$.
It is sufficient to show that $\Theta(n'K_4)\leq2m$ since $\Theta(nK_4)\leq \Theta(n'K_4)$ as $nK_4$ is an induced subgraph of $n'K_4$.
Consider the $\left(A, B\right)$-assignment of $n'K_4$.
By Lemma~\ref{EnotNEK4}$(i)$, the edge and nonedge rank sums do not coincide. 
Note that the set of edge rank sums of $n'K_4$ is $A\cup B$. Let $A\cup B=\left\lbrace c_1, c_2,\dots, c_m \right\rbrace$.
We separate the edge and nonedge rank sums by putting two thresholds around each edge rank sum.
For $i=1, 2,\dots, m$, let $\theta_{2i-1}=c_i$ and 
$\theta_{2i}=c_i+\varepsilon'$ be thresholds of $n'K_4$ where $\varepsilon'$ is a sufficiently small positive real number, for example, take $\varepsilon'$ smaller than any distance between two distinct rank sums of $n'K_4$.
Thus the above rank assignment is a $\left(\theta_{1}, \theta_{2}, \dots, \theta_{2m}\right)$-representation of $n'K_4$, and
hence $n'K_4$ is a $2m$-threshold graph, 
that is $\Theta(n'K_4)\leq2m$.

\noindent
\textbf{Case 2.} $m$ is odd. 

Let $n'=t_{m}-1=m+\binom{\left\lfloor{m/2}\right\rfloor}{3}+\binom{\left\lceil{m/2}\right\rceil}{3}$.
It is sufficient to show that $\Theta(n'K_4)\leq2m$ since $nK_4$ is an induced subgraph of $n'K_4$.
By Lemma~\ref{EnotNEK4}$(ii)$, 
there is a positive real number $\varepsilon$ such that, in the $\left(A, B, \varepsilon\right)$-assignment of $n'K_4$, 
no nonedge rank sum lies in either  $\left[a_i, a_i+\varepsilon\right]$, 
$\left[b_i, b_i+\varepsilon \right]$
or $\left\lbrace \frac{N}{2}+\varepsilon \right\rbrace$ for all $i\in \left\lbrace 1, 2,\dots, \left\lfloor{\frac{m}{2}}\right\rfloor\right\rbrace$, and 
moreover, the sets of the form $\left[a_i, a_i+\varepsilon\right]$, 
$\left[b_i, b_i+\varepsilon\right]$
and $\left\lbrace \frac{N}{2}+\varepsilon \right\rbrace$ for all $i\in \left\lbrace 1, 2,\dots, \left\lfloor{\frac{m}{2}}\right\rfloor\right\rbrace$ are pairwise disjoint.
Let $A\cup B\cup\left\lbrace \frac{N}{2}+\varepsilon \right\rbrace=\left\lbrace c_1, c_2,\dots, c_m \right\rbrace$.
We separate the edge and nonedge rank sums by putting two thresholds around each interval of edge rank sums of the form $\left[c_i, c_i+\varepsilon\right]$ and $\left\lbrace \frac{N}{2}+\varepsilon \right\rbrace$.
For $i=1, 2, \dots, m$, let $\theta_{2i-1}=c_i$ and
\begin{equation*}
	\theta_{2i}=
	\left\{
	\begin{array}{l}
		c_i+\varepsilon+\varepsilon' \hspace{.2cm}\text{if} \hspace{.2cm} c_i\in A\cup B,
		\vspace{.25cm}\\
		c_i+\varepsilon' \hspace{.8cm}\text{if} \hspace{.2cm} 
		c_i=\frac{N}{2}+\varepsilon
	\end{array}
	\right.
\end{equation*}	
be thresholds of $n'K_4$ where  $\varepsilon'$ is a sufficiently small positive real number.
Thus the above rank assignment is a $\left(\theta_1, \theta_2,\dots, \theta_{2m} \right)$-representation of $n'K_4$, and
hence $n'K_4$ is a $2m$-threshold graph, that is $\Theta(n'K_4)\leq2m$.

Suppose that $n=t_{m-1}$.
To prove that $\Theta(nK_4)\leq2m-1$, 
we write $M=\left\lfloor{\frac{m+1}{2}}\right\rfloor$ and 
let $\left\lbrace N, a_1, a_2,\dots,a_M\right\rbrace\subset \mathbb{R^+}$ be a linearly independent set over $\mathbb{Q}$ such that
$a_i<N\leq\frac{a_M}{2}$ for all $i\in[M-1]$.
Let 
$b_i=N-a_i$ for $i=1, 2,\dots, M-1$. 
Write $A=\left\lbrace a_1, a_2,\dots,a_{M-1}\right\rbrace$ and 
$B=\left\lbrace b_1, b_2,\dots,b_{M-1}\right\rbrace$.

\noindent
\textbf{Case 1.} $m-1$ is even. 

We take the $\left(A, B\right)$-assignment for the first $(m-1)+2\binom{(m-1)/2}{3}$ $K_4$'s in $nK_4$, and let every edge in the last $K_4$ have edge rank sum $a_M$.
Note that these $K_4$'s appear in the $\left(A\cup \left\lbrace a_M \right\rbrace, B\cup \left\lbrace b_M \right\rbrace\right)$-assignment of $\left(t_{m+1}-1\right)K_4$.
By Lemma~\ref{EnotNEK4}$(i)$, the edge and nonedge rank sums do not coincide.
Observe that the set of edge rank sums of $nK_4$ is $A\cup B\cup \left\lbrace a_M \right\rbrace$.
Let $A\cup B\cup \left\lbrace a_M \right\rbrace=\left\lbrace c_1, c_2,\dots, c_m \right\rbrace$
where $c_1<c_2<\dots<c_m$.
We separate the edge and nonedge rank sums by putting two thresholds around each edge rank sum.
For $i=1, 2,\dots, m$, let
$\theta_{2i-1}=c_i$ and 
$\theta_{2i}=c_i+\varepsilon'$ be thresholds of $nK_4$ where $\varepsilon'$ is a sufficiently small positive real number.
Thus the above rank assignment is a $\left(\theta_{1}, \theta_{2}, \dots, \theta_{2m}\right)$-representation of $nK_4$. 
In fact, we will show that we do not need the last threshold $\theta_{2m}$ by proving that no rank sum exceeds $\theta_{2m-1}$.
It is sufficient to show that the rank of each vertex is at most $\frac{\theta_{2m-1}}{2}=\frac{c_m}{2}=\frac{a_M}{2}$.
This is clear for the last $K_4$ with the set of edge rank sums $\left\lbrace a_M\right\rbrace$.
For the other $K_4$'s, the rank of each vertex is of the form  $\frac{c_i+c_j-c_k}{2}$ for some 
$i, j, k \in [m-1]$, 
which is at most $\frac{a_M}{2}$ since $c_i, c_j\leq\frac{a_M}{2}$ and $c_k>0$.
Thus the above rank assignment is a $\left(\theta_1, \theta_2,\dots, \theta_{2m-1} \right)$-representation of $nK_4$, and
hence $nK_4$ is a $(2m-1)$-threshold graph, that is $\Theta(nK_4)\leq2m-1$.

\noindent
\textbf{Case 2.} $m-1$ is odd. 

We choose $\varepsilon$ such that the $\left(A\cup \left\lbrace a_M \right\rbrace, B\cup \left\lbrace b_M \right\rbrace, \varepsilon\right)$-assignment of $\left(t_{m+1}-1\right)K_4$
satisfies the properties in Lemma~\ref{EnotNEK4}$(ii)$.
We then take the $\left(A, B, \varepsilon\right)$-assignment for the first $(m-1)+\binom{\left\lfloor(m-1)/2\right\rfloor}{3}+\binom{\left\lceil(m-1)/2\right\rceil}{3}$ $K_4$'s in $nK_4$, and let every edge in the last $K_4$ have edge rank sum $a_M$.
Note that these $K_4$'s appear in the $\left(A\cup \left\lbrace a_M \right\rbrace, B\cup \left\lbrace b_M \right\rbrace, \varepsilon\right)$-assignment of $\left(t_{m+1}-1\right)K_4$.
By the choice of $\varepsilon$, 
no nonedge rank sum lies in either  $\left[a_i, a_i+\varepsilon\right]$, 
$\left[b_i, b_i+\varepsilon \right]$
or $\left\lbrace \frac{N}{2}+\varepsilon \right\rbrace$ for all $i\in[M-1]$, and moreover, the sets of the form $\left[a_i, a_i+\varepsilon\right]$, 
$\left[b_i, b_i+\varepsilon\right]$
and $\left\lbrace \frac{N}{2}+\varepsilon \right\rbrace$ for all $i\in[M-1]$ are pairwise disjoint.
Let $A\cup B\cup\left\lbrace a_M,  \frac{N}{2}+\varepsilon \right\rbrace=\left\lbrace c_1, c_2,\dots, c_m \right\rbrace$ where $c_1<c_2<\dots<c_m$.
We claim that $c_m=a_M$.
Indeed, it is clear that $a_M>a_i, b_i$ for all $i\in [M-1]$.
Since $\frac{N}{2}+\varepsilon$ lies between the intervals $\left[a_1, a_1+\varepsilon\right]$ and  
$\left[b_1, b_1+\varepsilon \right]$ by the choice of $\varepsilon$, 
we have $\frac{N}{2}+\varepsilon<\max\left\lbrace a_1, b_1 \right\rbrace<a_M$.
We separate the edge and nonedge rank sums by putting two thresholds around each interval of edge rank sums.
For $i=1, 2, \dots, m$, let $\theta_{2i-1}=c_i$ and
\begin{equation*}
	\theta_{2i}=
	\left\{
	\begin{array}{l}
		c_i+\varepsilon+\varepsilon' \hspace{.2cm}\text{if} \hspace{.2cm} c_i\in A\cup B\cup\left\lbrace a_M \right\rbrace,
		\vspace{.25cm}\\
		c_i+\varepsilon' \hspace{.8cm}\text{if} \hspace{.2cm} 
		c_i=\frac{N}{2}+\varepsilon
	\end{array}
	\right.
\end{equation*}	
be thresholds of $nK_4$ where   $\varepsilon'$ is a sufficiently small positive real number.
Thus the above rank assignment is a $\left(\theta_1, \theta_2,\dots, \theta_{2m} \right)$-representation of $nK_4$.
In fact, we will show that we do not need the last threshold $\theta_{2m}$ by proving that no rank sum is greater than or equal to $\theta_{2m}=a_M+\varepsilon+\varepsilon'$.
It is sufficient to show that the rank of each vertex is at most $\frac{a_M+\varepsilon}{2}$.
This is clear for the last $K_4$ with the set of edge rank sums $\left\lbrace a_M \right\rbrace$.
For the other $K_4$'s, the rank of each vertex is of the form 
$\frac{d_i+d_j-d_k}{2},
\frac{N}{4}+\frac{\varepsilon}{2}$ 
or 
$\frac{d_i+d_j-N/2+\varepsilon}{2}$ where 
$i, j, k \in [M-1]$ and
$d_\ell\in \left\lbrace a_\ell, b_\ell \right\rbrace$
for $\ell\in\left\lbrace i, j, k \right\rbrace$, which is at most $\frac{a_M+\varepsilon}{2}$ 
since $0<d_i, d_j, d_k, \frac{N}{2}\leq\frac{a_M}{2}$.
Thus the above rank assignment is a $\left(\theta_1, \theta_2,\dots, \theta_{2m-1} \right)$-representation of $nK_4$, and hence
$nK_4$ is a $(2m-1)$-threshold graph, that is $\Theta(nK_4)\leq2m-1$.

Suppose that $n>t_{m-1}$.
To prove that $\Theta(nK_4)\geq2m$, we suppose that $\Theta(nK_4)\leq2m-1$.
Let $r$ be a $\left(\theta_1, \theta_2,\dots, \theta_{2m-1} \right)$-representation of $nK_4$.
Then there are at most $m$ colors of edges in $nK_4$.
By Lemma~\ref{lemK4}$(i)$, there are at most $t_{m-1}-1$ $K_4$'s without color $m$.  
By Lemma~\ref{nocolorK4}$(i)$, an edge of color $m$ appears in at most one $K_4$.
Thus $n\leq \left(t_{m-1}-1\right)+1$, a contradiction.
Therefore, $\Theta(nK_4)\geq2m$.
\end{proof}

%%%%%%%%%%%%%%%%%%%%%%%%%%%%%%%%%%%%%%%%%%%
%%%%%%%%%%%%%%%%%%%%%%%%%%%%%%%%%%%%%%%%%%%
%\begin{thm} \label{TK44...4}
%Let $S_k=k+\binom{\left\lfloor{k/2}\right\rfloor}{3}+\binom{\left\lceil{k/2}\right\rceil}{3}+2$. 
%For $n\geq 2$, 
%	\begin{equation*}
%		\Theta(K_{n\times4})=
%		\left\{
%		\begin{array}{l}
%			2m \hspace{.8cm}\text{if} \hspace{.2cm} n=S_{m-1},
%			\\
%			2m+1 \hspace{.2cm}\text{if} \hspace{.2cm} S_{m-1} < n < S_m.
%		\end{array}
%		\right.
%	\end{equation*}	
%\end{thm}

\begin{proof} [Proof of Theorem~\ref{Main theorem K4}$(ii)$]
Let $m$ be such that $s_{m-1} \leq n < s_m$.
By Theorem~\ref{Main theorem K4}$(i)$, 
$\Theta(nK_4)\in \left\lbrace 2m,  2m+1\right\rbrace$, and hence
$\Theta(K_{n\times4})\in \left\lbrace 2m,  2m+1\right\rbrace$ by Proposition~\ref{complement}.

Suppose that $n=s_{m-1}$.
To prove that $\Theta(K_{n\times4})\leq2m$, 
we write $M=\left\lfloor{\frac{m+1}{2}}\right\rfloor$ and
let $\left\lbrace N, a_1, a_2,\dots,a_M\right\rbrace\subset \mathbb{R}$ be a linearly independent set over $\mathbb{Q}$ such that
$\frac{a_M}{3}\leq -N < -a_i < 0$ for all $i\in[M-1]$.  
Let $b_i=N-a_i$ for $i= 1, 2, \dots, M$.
Then
$\frac{a_M}{3}\leq a_i, b_i, -N, N\leq\frac{b_M}{3}$ 
for all $i\in[M-1]$.
Write 
$A=\left\lbrace a_1, a_2,\dots,a_{M-1}\right\rbrace$ and 
$B=\left\lbrace b_1, b_2,\dots,b_{M-1}\right\rbrace$.

\noindent
\textbf{Case 1.} $m-1$ is even. 

We take the $\left(A, B\right)$-assignment for the first $(m-1)+2\binom{(m-1)/2}{3}$ parts in $K_{n\times4}$, and 
let the last two parts have the sets of nonedge rank sums $\left\lbrace a_M\right\rbrace$ and $\left\lbrace b_M\right\rbrace$.
Note that these parts appear in the $\left(A\cup\left\lbrace a_M \right\rbrace, B\cup\left\lbrace b_M \right\rbrace\right)$-assignment of $K_{(s_{m+1}-2)\times4}$.
By Lemma~\ref{EnotNEK4}$(iii)$, the edge and nonedge rank sums do not coincide.
Observe that the set of nonedge rank sums of $K_{n\times4}$ is $A\cup B\cup\left\lbrace a_M, b_M \right\rbrace$.
Let $A\cup B\cup\left\lbrace a_M, b_M \right\rbrace=\left\lbrace c_1, c_2, \dots, c_{m+1} \right\rbrace$ where $c_1<c_2<\dots<c_{m+1}$.
Let $\theta_1$ be smaller than all rank sums.
We then separate the edge and nonedge rank sums by putting two thresholds around each nonedge rank sum.
For $i=1, 2, \dots, m+1$, let $\theta_{2i}=c_i$ and $\theta_{2i+1}=c_i+\varepsilon'$ where $\varepsilon'$ is a sufficiently small positive real number. 
Thus the above rank assignment is a $\left(\theta_1, \theta_2,\dots, \theta_{2m+3} \right)$-representation of $K_{n\times4}$.
In fact, we will show that we do not need the thresholds $\theta_1, \theta_2$ and $\theta_{2m+3}$ by proving that no rank sum is smaller than $\theta_2$ or larger than $\theta_{2m+2}$.
It is sufficient to show that the rank of each vertex is at least $\frac{\theta_{2}}{2}=\frac{c_1}{2}=\frac{a_M}{2}$ and
at most $\frac{\theta_{2m+2}}{2}=\frac{c_{m+1}}{2}=\frac{b_M}{2}$.
This is clear for the last two parts with the sets of nonedge rank sums $\left\lbrace a_M \right\rbrace$ and $\left\lbrace b_M \right\rbrace$.
For the other parts, 
the rank of each vertex is of the form $\frac{c_i+c_j-c_k}{2}$ for some 
$i, j, k \in [m]\setminus \left\lbrace 1 \right\rbrace$, 
which is at least $\frac{a_M}{2}$ and 
at most $\frac{b_M}{2}$ since  $\frac{a_M}{3}\leq c_i, c_j, -c_k\leq\frac{b_M}{3}$.
Thus the above rank assignment is a $\left(\theta_3, \theta_4,\dots, \theta_{2m+2} \right)$-representation of $K_{n\times4}$, and 
hence $K_{n\times4}$ is a $2m$-threshold graph, that is $\Theta(K_{n\times4})\leq2m$.

\noindent
\textbf{Case 2.} $m-1$ is odd. 

We choose $\varepsilon$ such that the $\left(A\cup \left\lbrace a_M \right\rbrace, B\cup \left\lbrace b_M \right\rbrace, \varepsilon\right)$-assignment of $K_{(s_{m+1}-2)\times4}$
satisfies the properties in Lemma~\ref{EnotNEK4}$(iv)$.
We then take the $\left(A, B, \varepsilon\right)$-assignment for the first $(m-1)+\binom{\left\lfloor(m-1)/2\right\rfloor}{3}+\binom{\left\lceil(m-1)/2\right\rceil}{3}$ parts in $K_{n\times4}$, and let  the last two parts have the sets of nonedge rank sums $\left\lbrace a_M \right\rbrace$ and $\left\lbrace b_M \right\rbrace$.
Note that these parts appear in the $\left(A\cup \left\lbrace a_M \right\rbrace, B\cup \left\lbrace b_M \right\rbrace, \varepsilon\right)$-assignment of $K_{(s_{m+1}-2)\times4}$.
By the choice of $\varepsilon$, 
no edge rank sum lies in either  $\left[a_i, a_i+\varepsilon\right]$, 
$\left[b_i, b_i+\varepsilon \right]$
or $\left\lbrace \frac{N}{2}+\varepsilon \right\rbrace$ for all $i\in [M-1]$, and moreover, the sets of the form $\left[a_i, a_i+\varepsilon\right]$, 
$\left[b_i, b_i+\varepsilon\right]$
and $\left\lbrace \frac{N}{2}+\varepsilon \right\rbrace$ for all $i\in [M-1]$ are pairwise disjoint.
Let $A\cup B\cup\left\lbrace a_M, b_M,  \frac{N}{2}+\varepsilon \right\rbrace=\left\lbrace c_1, c_2, \dots, c_{m+1} \right\rbrace$ where $c_1<c_2<\dots<c_{m+1}$.
We claim that $c_1=a_M$ and $c_{m+1}=b_M$.
Indeed, it is clear that 
$a_M<a_i, b_i<b_M$ for all $i\in [M-1]$.
Since $\frac{N}{2}+\varepsilon$ lies between the intervals $\left[a_1, a_1+\varepsilon\right]$ and  
$\left[b_1, b_1+\varepsilon \right]$ by the choice of $\varepsilon$, 
we have $a_M<\min\left\lbrace a_1, b_1 \right\rbrace<\frac{N}{2}+\varepsilon<\max\left\lbrace a_1, b_1 \right\rbrace<b_M$.
Let $\theta_1$ be smaller than all rank sums.
We then separate the edge and nonedge rank sums by putting two thresholds around each interval of nonedge rank sums.
For $i=1, 2, \dots, m+1$, let $\theta_{2i}=c_i$ and
\begin{equation*}
	\theta_{2i+1}=
	\left\{
	\begin{array}{l}
		c_i+\varepsilon+\varepsilon' \hspace{.2cm}\text{if} \hspace{.2cm} c_i\in A\cup B\cup\left\lbrace a_M, b_M \right\rbrace,
		\vspace{.25cm}\\
		c_i+\varepsilon' \hspace{.8cm}\text{if} \hspace{.2cm} 
		c_i=\frac{N}{2}+\varepsilon
	\end{array}
	\right.
\end{equation*}	
be thresholds of $K_{n\times4}$ 
where $\varepsilon'$ is a sufficiently small positive real number.
Thus the above rank assignment is a $\left(\theta_1, \theta_2,\dots, \theta_{2m+3} \right)$-representation of $K_{n\times4}$.
In fact, we will show that we do not need the thresholds $\theta_1, \theta_2$ and $\theta_{2m+3}$ by proving that no rank sum is smaller than $\theta_2$, or larger than or equal to $\theta_{2m+3}$.
It is sufficient to show that the rank of each vertex is at least $\frac{\theta_2}{2}=\frac{c_1}{2}=\frac{a_M}{2}$ and
at most $\frac{\theta_{2m+3}-\varepsilon'}{2}=\frac{c_{m+1}+\varepsilon}{2}=\frac{b_M+\varepsilon}{2}$.
This is clear for the last two parts with the sets of nonedge rank sums $\left\lbrace a_M \right\rbrace$ and $\left\lbrace b_M \right\rbrace$.
For the other parts, the rank of each vertex is of the form $\frac{d_i+d_j-d_k}{2}, \frac{N}{4}+\frac{\varepsilon}{2}$ or $\frac{d_i+d_j-N/2+\varepsilon}{2}$ where
$i, j, k \in [M-1]$ are all distinct and 
$d_\ell\in \left\lbrace a_\ell, b_\ell \right\rbrace$ for $\ell\in \left\lbrace i, j, k \right\rbrace$, which is at least $\frac{a_M}{2}$ and at most $\frac{b_M+\varepsilon}{2}$ since  $\frac{a_M}{3}\leq d_i, d_j, -d_k, -N, N\leq\frac{b_M}{3}$.
Thus the above rank assignment is a $\left(\theta_3, \theta_4,\dots, \theta_{2m+2} \right)$-representation of $K_{n\times4}$, and
hence, $K_{n\times4}$ is a $2m$-threshold graph, that is $\Theta(K_{n\times4})\leq2m$.

Suppose that $n>s_{m-1}$.	
To prove that $\Theta(K_{n\times4})\geq2m+1$, we suppose that $\Theta(K_{n\times4})\leq2m$.
Let $r$ be a $\left(\theta_1, \theta_2,\dots, \theta_{2m} \right)$-representation of $K_{n\times4}$.
Then there are at most $m+1$ colors of nonedges in $K_{n\times4}$.
By Lemma~\ref{lemK4}$(ii)$, 
there are at most $s_{m-1}-2$ parts without colors $1$ and $m+1$.
By Lemma~\ref{nocolorK4}$(ii)$ and~\ref{nocolorK4}$(iii)$, a nonedge of color $1$ appears in at most one part and a nonedge of color $m+1$ also appears in at most one part.
Therefore, $n\leq (s_{m-1}-2)+1+1$, a contradiction.
\end{proof}
 
%%%%%%%%%%%%%%%%%%%%%%%%%%%%%%%%%%%%%%%%%%%%%%%%%%%%%%%%%%%%%%%%%%%%%%%%%%%%%%%%%%%%%%
 
\section{Concluding Remarks} \label{section-conclusion}
We find the threshold numbers of $K_{n\times3}$ and $K_{n\times4}$, while
Chen and Hao~\cite{CH} determined that of $K_{m_1, m_2, \dots, m_n}$ for $m_i>n\geq 2$.
Problem~\ref{prob} remains unsolved for other complete multipartite graphs.
The following could be the next goal.
\begin{prob}
	Determine the exact threshold number of $K_{n\times m}$ for $m\geq 5$.
\end{prob}
\noindent
The method we used can be generalized to give some bounds for $\Theta(K_{n\times m})$, but new ideas seem to be required in order to find the exact value.
%%%%%%%%%%%%%%%%%%%%%%%%%%%%%%%%%%%%%%%%%%%%%%%%%%%%%%%%%%%%%%%%%%%%%%%%%%%%%%%%%%%%%%
	
%\section*{Acknowledgment}	
%
%	The second author is grateful for financial support from the Development and Promotion of Science and Technology Talents Project. 

%	\bibliographystyle{siam} 
%	\bibliography{GGG}

\end{document}